\documentclass{amsart}
\usepackage{amssymb, amsmath, amsthm, graphicx, amscd, pxfonts, mathtools, tikz-cd}
\usepackage[scr]{rsfso}

\usepackage{array,longtable}
\newcolumntype{C}{>{$}c<{$}}  
\newcolumntype{L}{>{$}l<{$}} 
\theoremstyle{plain}
\newtheorem{thm}{Theorem}
\newtheorem{lm}[thm]{Lemma}
\newtheorem{cor}[thm]{Corollary}
\newtheorem{prop}[thm]{Proposition}

\newtheorem*{main}{Main Theorem}

\theoremstyle{definition}

\newtheorem{de}[thm]{Definition}

\newtheorem{exs}[thm]{Examples}

\newtheorem{re}[thm]{Remark}

\DeclareMathOperator{\NN}{\mathbb{N}}
\DeclareMathOperator{\ZZ}{\mathbb{Z}}

\DeclareMathOperator{\PP}{\mathbb{P}}

\DeclareMathOperator{\A}{A}
\DeclareMathOperator{\B}{B}
\DeclareMathOperator{\C}{C}
\DeclareMathOperator{\D}{D}

\DeclareMathOperator{\Sym}{Sym}

\DeclareMathOperator{\Diag}{Diag}
\DeclareMathOperator{\Gr}{Gr}

\DeclareMathOperator{\id}{id}
\DeclareMathOperator{\rk}{rk}
\DeclareMathOperator{\cha}{char}
\DeclareMathOperator{\tr}{tr}
\let\span\relax\DeclareMathOperator{\span}{span}
\let\mod\relax\DeclareMathOperator{\mod}{~mod~}
\DeclareMathOperator{\pr}{pr}
\DeclareMathOperator{\grad}{grad}

\DeclareMathOperator{\GL}{GL}
\DeclareMathOperator{\SL}{SL}
\let\O\relax\DeclareMathOperator{\O}{O}
\DeclareMathOperator{\Sp}{Sp}

\DeclareMathOperator{\g}{\mathfrak{g}}
\DeclareMathOperator{\h}{\mathfrak{h}}
\DeclareMathOperator{\gl}{\mathfrak{gl}}
\let\sl\relax\DeclareMathOperator{\sl}{\mathfrak{sl}}
\let\o\relax\DeclareMathOperator{\o}{\mathfrak{o}}
\let\sp\relax\DeclareMathOperator{\sp}{\mathfrak{sp}}

\makeatletter
\def\iddots{\mathinner{\mkern1mu\raise\p@
    \vbox{\kern7\p@\hbox{.}}\mkern2mu
    \raise4\p@\hbox{.}\mkern2mu\raise7\p@\hbox{.}\mkern1mu}}
\makeatother

\begin{document}
\title{Noetherianity up to conjugation of locally diagonal inverse limits} 
\author{Arthur Bik}
\address{Universit\"at Bern, Mathematisches Institut, Alpeneggstrasse 22,
3012 Bern, Switzerland}
\email{arthur.bik@math.unibe.ch}
\begin{abstract}
\textit{We prove that the inverse limit of the sequence dual to a sequence of Lie algebras is Noetherian up to the action of the direct limit of the corresponding sequence of classical algebraic groups when the sequence of groups consists of diagonal embeddings. We also classify all conjugation-stable closed subsets of the space of $\NN\times\NN$ matrices.}
\end{abstract}
\thanks{The author was partially supported by the NWO Vici grant entitled {\em Stabilisation in Algebra and Geometry}.}
\maketitle

Throughout this paper, we work over an infinite field $K$. Consider a sequence of groups
$$\begin{tikzcd}
G_1 \arrow{r} & G_2 \arrow{r} & G_3 \arrow{r} & \dots
\end{tikzcd}$$
together with a sequence of finite-dimensional vector spaces over $K$
$$\begin{tikzcd}
V_1 & \arrow{l} V_2 & \arrow{l} V_3 & \arrow{l} \dots
\end{tikzcd}$$
 such that $V_i$ is a representation of $G_i$ and the map $V_{i+1}\rightarrow V_i$ is $G_i$-equivariant for all $i\in\NN$. Then the direct limit $G$ of the sequence of groups naturally acts on the inverse limit $V$ of the sequence of vector spaces. A subset $X$ of $V$ is Zariski-closed if it is the inverse limit of a sequence of Zariski-closed subsets $X_i\subseteq V_i$. Now one can ask the following question. Given a descending sequence
$$
V\supseteq X^{(1)}\supseteq X^{(2)}\supseteq X^{(3)}\supseteq \dots
$$
of Zariski-closed $G$-stable subsets of $V$, is there always a $j\in\NN$ such that $X^{(i)}=X^{(j)}$ for all $i\geq j$?\bigskip

If the answers is yes, then the space $V$ is called $G$-Noetherian. See \cite{hillar-sullivant,draisma-eggermont,eggermont} for examples of such spaces. The easiest example of a space $V$ that is not \mbox{$G$-Noetherian} is given by an infinite-dimenional vector space acted on by the trivial group. Recently it was proven \cite{draisma} that polynomial functors of finite degree are Noetherian. Such functors give rise to $G$-Noetherian spaces $V$ where $G_i=\GL_i$, the map $G_i\to G_{i+1}$ is given by
$$
A\mapsto\begin{pmatrix}A\\&1\end{pmatrix}
$$
and where $V_i$ is a polynomial representation of $\GL_i$. This was then generalised \cite{eggermont-snowden} to algebraic polynomial functors of finite degree. Such functors give sequences $(G_i)_{i\geq1}$ of classical algebraic groups together with algebraic representations $(V_i)_{i\geq1}$.\bigskip

In this paper, we consider sequence of classical algebraic groups that do not arise this way, such as the sequence
$$\begin{tikzcd}
\SL_1 \arrow[hook]{r} & \SL_2 \arrow[hook]{r} & \SL_4 \arrow[hook]{r} & \dots\arrow[hook]{r}&\SL_{2^i}\arrow[hook]{r}&\dots
\end{tikzcd}$$
with maps given by
\begin{eqnarray*}
\SL_{2^i}&\hookrightarrow&\SL_{2^{i+1}}\\
A&\mapsto&\begin{pmatrix}A\\&A\end{pmatrix},
\end{eqnarray*}
where the image of an element $A\in G_i$ in $G_{i+1}$ can contain multiple copies of $A$. To such a sequence of groups, there is a corresponding sequence of Lie algebras, which we then dualize to get a sequence going in the opposite direction. We prove that the inverse limit of this sequence is Noetherian up to the action of the direct limit of the sequence of groups.

\subsection*{Notation and conventions}
Let $\NN$ be the set of positive integers. Denote the dual of a vector space $V$ by $V^*$. Let $i,j,k,\ell,m,n\in\NN$ be integers. Define $\delta_{ij}$ to be $1$ if $i=j$ and $0$ if $i\neq j$. Denote the set of $n\times n$ matrices by $\gl_n$. When $m\leq n$, we write $\pr_m$ for the projection map $\gl_n\twoheadrightarrow\gl_m$ of $n\times n$ matrices onto their topleft $m\times m$ submatrix. Denote the inverse limit of the sequence
$$\begin{tikzcd}
\gl_{1} & \arrow[two heads]{l} \gl_{2} & \arrow[two heads]{l} \gl_{3} & \arrow[two heads]{l} \dots
\end{tikzcd}$$
by $\gl_\infty$, let $I_\infty\in\gl_\infty$ be the infinite identity matrix and write $\pr_n$ for the projection map $\gl_\infty\twoheadrightarrow\gl_n$. Denote the set $\{1,\dots,n\}$ by $[n]$. Let $P,Q\in\gl_n$ be matrices. For subsets $\mathscr{K},\mathscr{L}\subseteq[n]$, we write $P_{\mathscr{K},\mathscr{L}}$ for the submatrix of $P$ with rows $\mathscr{K}$ and columns $\mathscr{L}$. We say that $P$ and $Q$ are similar (and write $P\sim Q$) if there is a matrix $A\in\GL_n$ such that $P=AQA^{-1}$. We say that $P$ and $Q$ are congruent if there is a matrix $B\in\GL_n$ such that $P=BQB^T$. For matrices $P_1,\dots,P_k$ not necessarily of the same size, denote the block-diagonal matrix with blocks $P_1,\dots,P_k$ by $\Diag(P_1,\dots,P_k)$. 

\subsection*{Acknowledgements}
I thank Jan Draisma and Micha\l{} Laso\'{n} for the helpful discussions I had with them. I also thank Jan Draisma for finding and proving Proposition \ref{prop_tuplerank} and for proofreading this paper. Finally, I thank the anonymous referee  for carefully reading this paper and for their useful comments.

\section{The main results}

We consider sequences of embeddings
$$\begin{tikzcd}
G_1 \arrow[hook,"\iota_1"]{r} & G_2 \arrow[hook,"\iota_2"]{r} & G_3 \arrow[hook,"\iota_3"]{r} & \dots
\end{tikzcd}$$
built up out of homomorphisms between the following classical algebraic groups
\begin{longtable}{CLCLCL}
\A_{n-1}:&\SL_n&=&\{A\in\GL_n\mid\det(A)=1\}&\mbox{for }n\in\NN\\~\\
\B_n:&\O_{2n+1}&=&\left\{A\in\GL_{2n+1}~\middle|~A\begin{pmatrix}&&I_n\\&1\\I_n\end{pmatrix}A^T=\begin{pmatrix}&&I_n\\&1\\I_n\end{pmatrix}\right\}&\mbox{for }n\in\NN\\~\\
\C_n:&\Sp_{2n}&=&\left\{A\in\GL_{2n}~\middle|~A\begin{pmatrix}&I_n\\-I_n\end{pmatrix}A^T=\begin{pmatrix}&I_n\\-I_n\end{pmatrix}\right\}&\mbox{for }n\in\NN\\~\\
\D_n:&\O_{2n}&=&\left\{A\in\GL_{2n}~\middle|~A\begin{pmatrix}&I_n\\I_n\end{pmatrix}A^T=\begin{pmatrix}&I_n\\I_n\end{pmatrix}\right\}&\mbox{for }n\in\NN
\end{longtable} 
\noindent which we view as embedded subgroups of $\GL_n$, for appropriate $n\in\NN$. Let $G,H$ be such groups, let $V,W$ be their standard representations and consider $K$ as the trivial representation of $G$. In \cite{baranov-zhilinskii}, an embedding $G\hookrightarrow H$ is called diagonal if
$$
W\cong V^{\oplus l}\oplus (V^*)^{\oplus r}\oplus K^{\oplus z}
$$
as representations of $G$ for some $l,r,z\in\ZZ_{\geq0}$ with $l+r\geq1$. The triple $(l,r,z)$ is called the signature of the embedding. If $G$ is of type $\A$, then the signature of a diagonal embedding $G\hookrightarrow H$ is unique. However, if $G$ is of type $\B$, $\C$ or $\D$, then the representation $V$ is isomorphic to $V^*$. In this case, we will always assume that $r=0$, which makes the pair $(l,z)$ unique, and we also denote the signature by $(l,z)$.

\begin{exs}
Let $G\subseteq\GL_n,H,L$ be classical groups of type $\A$, $\B$, $\C$ or $\D$.
\begin{itemize}
\item[(a)] For each $B\in\GL_n$ with $BG=GB$, the automorphism 
\begin{eqnarray*}
G&\rightarrow&G\\
A&\mapsto& BAB^{-1}
\end{eqnarray*}
is diagonal with signature $(1,0,0)$.
\item[(b)] For all matrices $A\in G$, we have $A^{-T}\in G$. The automorphism
\begin{eqnarray*}
G&\rightarrow&G\\
A&\mapsto& A^{-T}
\end{eqnarray*}
is diagonal with signature $(0,1,0)$.
\item[(c)] The composition of any two diagional embeddings $G\hookrightarrow H$ and $H\hookrightarrow L$ is a diagonal embedding $G\hookrightarrow L$. 
\end{itemize}
\end{exs}

We will assume the sequence
$$\begin{tikzcd}
G_1 \arrow[hook,"\iota_1"]{r} & G_2 \arrow[hook,"\iota_2"]{r} & G_3 \arrow[hook,"\iota_3"]{r} & \dots
\end{tikzcd}$$
consists of diagonal embeddings. Let $G$ be its direct limit and consider the associated sequence
$$\begin{tikzcd}
\g_1 \arrow[hook]{r} & \g_2 \arrow[hook]{r} & \g_3 \arrow[hook]{r} & \dots
\end{tikzcd}$$
where $\g_i$ is the Lie algebra of $G_i$. Now, we let $V$ be the inverse limit of the sequence
$$\begin{tikzcd}
\g_1^* & \arrow[two heads]{l} \g_2^* & \arrow[two heads]{l} \g_3^* & \arrow[two heads]{l} \dots
\end{tikzcd}$$
obtained by dualizing the previous sequence. Then $V$ has a natural action of $G$. If we modify our sequence by replacing  
$$\begin{tikzcd}
G_i \arrow[hook,"\iota_i"]{r} & G_{i+1} \arrow[hook,"\iota_{i+1}"]{r} & G_{i+2} 
\end{tikzcd}$$
by
$$\begin{tikzcd}
G_i \arrow[hook,"\iota_{i+1}\circ\iota_i"]{rr} &\quad\quad& G_{i+2} 
\end{tikzcd}$$
then both the direct limit $G$ and the inverse limit $V$ do not change. So we may replace our sequence of groups by any of its infinite subsequences. Conversely, we can also replace our sequence by any supersequence. Note that there always exists an infinite subsquence such that every group in the subsequence is of the same type.

\begin{main}\label{main}
Assume that one of the following conditions hold:
\begin{itemize}
\item[(a)] The group $G_i$ has type $\A$ for infinitely many $i\in\NN$.
\item[(b)] The characteristic of $K$ does not equal $2$.
\end{itemize}
Then the space $V$ is $G$-Noetherian, i.e. for every descending sequence
$$
V\supseteq X_1\supseteq X_2\supseteq X_3\supseteq\dots
$$
of $G$-stable closed subsets of $V$ there is an $i\in\NN$ such that $X_i=X_j$ for all $j\geq i$.
\end{main}

\begin{re}
When we prove the Main Theorem, we may assume that all $G_i$ have the same type. When this type is $\B$, $\C$ or $\D$, we assume that $\cha(K)\neq2$. This way we know that the set of (skew-)symmetric $n\times n$ matrices congruent to some given (skew-)symmetric matrix $A$ equals the set of all (skew-)symmetric matrices whose rank is equal to the rank of $A$. See the proofs of Lemmas \ref{lm_boundedranksp} and \ref{lm_boundedrankod} and Proposition~\ref{prop_boundedrankd}.
\end{re}

When all $G_i$ are of type $\A$ and $(l,r)=(1,0)$ for all but finitely many embeddings, the group $G$ equals $\SL_\infty$ and the space $V$ can be identified with a quotient of the set $\gl_\infty$ of $\NN\times\NN$ matrices. We prove this case of the Main Theorem by classifying all $\SL_\infty$-stable closed subsets of $\gl_\infty$.

\begin{de}
Define the rank of a matrix $P\in\gl_\infty$ as 
$$
\rk(P)=\sup\{\rk(\pr_n(P))\mid n\in\NN\}\in\ZZ_{\geq0}\cup\{\infty\}.
$$
\end{de}

We use the following definition from \cite{draisma-eggermont}.

\begin{de}
Let $n\in\NN\cup\{\infty\}$ and let $Q_1,\dots,Q_k$ be elements of $\gl_n$. Define 
$$
\rk(Q_1,\dots,Q_k)=\inf\left\{\rk(\mu_1Q_1+\dots+\mu_kQ_k)~\middle|~(\mu_1:\dots:\mu_k)\in\PP^{k-1}\right\}\in\ZZ_{\geq0}\cup\{\infty\}
$$
to be the rank of the tuple $(Q_1,\dots,Q_k)$.
\end{de}

\begin{thm}\label{thm_maingl}
The space $\gl_{\infty}$ is $\SL_\infty$-Noetherian. Any $\SL_\infty$-stable closed subset of $\gl_{\infty}$ is a finite union of irreducible $\SL_\infty$-stable closed subsets and the irreducible $\SL_\infty$-stable closed subsets of $\gl_\infty$ are $\gl_\infty$ itself together with the subsets
$$
\left\{P\in\gl_\infty~\middle|~\rk(P,I_\infty)\leq k\right\}, \left\{P\in\gl_\infty~\middle|~\rk(P-\lambda I_\infty)\leq k\right\}
$$ 
for $\lambda\in K$ and $k\in\ZZ_{\geq0}$.
\end{thm}

\begin{re}
We would like to point out that the $\SL_\infty$-Noetherianity of $\gl_\infty$ also follows from \cite[Theorem 1.2]{eggermont-snowden}.
\end{re}

\begin{re}\label{re_uniquehighgeomult}
Let $P\in\gl_{\infty}$ be an $\NN\times\NN$ matrix such that $\rk(P,I_\infty)<\infty$. Then we have $\rk(P-\lambda I_\infty)<\infty$ for some $\lambda\in K$. If this holds for distinct $\lambda,\lambda'\in K$, then
$$
\infty=\rk((\lambda-\lambda')I_\infty)=\rk\left((P-\lambda'I_\infty)-(P-\lambda I_\infty)\right)\leq \rk(P-\lambda'I_\infty)+\rk(P-\lambda I_\infty)<\infty
$$
and hence the $\lambda\in K$ such that $\rk(P-\lambda I_\infty)<\infty$ must be unique. This is the infinite analogue of the statement that an $n\times n$ matrix can have at most one eigenvalue with geometric multiplicity more than $n/2$.
\end{re}

\begin{re}
When we call each of the closed subsets $X\subseteq\gl_\infty$ listed in the theorem irreducible, we mean this in the following sense: if we have 
$$
X=Y\cup Z
$$
for (not necessarily $\SL_\infty$-stable) closed subsets $Y,Z\subseteq X$, then $X=Y$ or $X=Z$.
\end{re}

\section{Structure of the proof}

In this section, we reduce the Main Theorem to a number of cases and we outline the structure that the proofs of each of those cases share.

\subsection{Reduction to standard diagonal embeddings}
When the vector space $V$ is finite-dimesional over $K$, the Main Theorem becomes trivial. So we will only consider the cases where $V$ is infinite-dimensional. For all $i\in\NN$, let $(l_i,r_i,z_i)$ be the signature of the embedding $\iota_i\colon G_i\hookrightarrow G_{i+1}$. When $G_i$ is of type $\B$, $\C$ or $\D$, we will assume that $r_i=0$. The following lemma tells us that we can assume that $l_i\geq r_i$ for all $i\in\NN$.

\begin{lm}
For all $i\in\NN$, let $\sigma_i\colon G_i\rightarrow G_i$ be the automorphism sending $A\mapsto A^{-T}$ and take $k_i\in\ZZ/2\ZZ$. Then the bottom row of the commutative diagram
$$\begin{tikzcd}
G_1 \arrow[hook,"\iota_1"]{r}\arrow["\sigma_1^{k_1}"]{d} & G_2 \arrow[hook,"\iota_2"]{r}\arrow["\sigma^{k_2}_2"]{d} & G_3 \arrow[hook,"\iota_3"]{r}\arrow["\sigma^{k_3}_3"]{d} & \dots\\
G_1 \arrow[hook]{r} & G_2 \arrow[hook]{r} & G_3 \arrow[hook]{r} & \dots
\end{tikzcd}$$
is a sequence of diagonal embeddings with signatures $\sigma^{k_i+k_{i+1}}(l_i,r_i,z_i)$ where $\sigma$ acts by permuting the first two entries.
\end{lm}

The lemma follows from the fact that the automorphism $G_i\rightarrow G_i, A\mapsto A^{-T}$ is diagonal and its own inverse. We can choose the $k_i$ recursively so that $l_i\geq r_i$ for all $i\in\NN$ in the bottom sequence. Since the vertical maps are isomorphisms and the diagram commutes, the bottom sequence gives rise to isomorphic $G$ and $V$. This allows us to indeed assume that $l_i\geq r_i$.\bigskip

Let $G$ be a classical group of type $\A$, $\B$, $\C$ or $\D$. Let $l,r,z\in\ZZ_{\geq0}$ be integers with $r=0$ if $G$ is not of type~$\A$. Let $\beta_1,\beta_2$ be non-degenerate $G$-invariant bilinear forms on $V^{\oplus l}\oplus (V^*)^{\oplus r}\oplus K^{\oplus z}$.

\begin{lm}
Assume that $K=\overline{K}$ and that one of the following conditions hold:
\begin{itemize}
\item[(a)] $\beta_1$ and $\beta_2$ are both skew-symmetric.
\item[(b)] $\beta_1$ and $\beta_2$ are both symmetric and $\cha(K)\neq2$.
\end{itemize}
Then there exists a $G$-equivariant automorphism $\varphi$ of $V^{\oplus l}\oplus (V^*)^{\oplus r}\oplus K^{\oplus z}$ such that 
$$
\beta_2(\varphi(v),\varphi(w))=\beta_1(v,w)
$$
for all $v,w\in V^{\oplus l}\oplus (V^*)^{\oplus r}\oplus K^{\oplus z}$.
\end{lm}
\begin{proof}
First suppose that $l=r=0$. In this case, the lemma reduces to the well-known statement that the matrices corresponding to $\beta_1$ and $\beta_2$ are congruent. In genenal, Schur's Lemma splits the lemma into the cases $r=z=0$, $l=z=0$ and $l=r=0$ and reduces the first two cases to the third.
\end{proof}

Let $f,g\colon G\rightarrow H\subseteq\GL_n$ be two diagonal embeddings with signature $(l,r,z)$.

\begin{lm}\label{lm_equiv_diag}
If the type of $H$ is $\B$, $\C$ or $\D$, assume that $K=\overline{K}$. If the type of $H$ is $\B$ or $\D$, assume in addition that $\cha(K)\neq2$. Then there is a $P\in H$ such that the isomorphism $\pi\colon H\rightarrow H,A\mapsto PAP^{-1}$ makes the diagram
$$\begin{tikzcd}
  G  \arrow[hook,"f"]{r} \arrow["\id"]{d} & H \arrow["\pi"]{d} \\
  G \arrow[hook,"g"]{r} & H
\end{tikzcd}$$
commute.
\end{lm}
\begin{proof}
The maps $f$ and $g$ both induce an isomorphism 
$$
K^n\cong V^{\oplus l}\oplus (V^*)^{\oplus r}\oplus K^{\oplus z}
$$
of representations of $G$. This means that there are matrices $Q,R$ such that
$$
Qf(A)Q^{-1}=Rg(A)R^{-1}=\Diag(A,\dots,A,A^{-T},\dots,A^{-T},I_z)
$$
for all $A\in G$ where the block-diagonal matrix has $l$ blocks $A$ and $r$ blocks $A^{-T}$. If $H$ is of type $\A$, then we take $P=\lambda R^{-1}Q$ for some $\lambda\in K$ such that $P\in\SL_n$ and see that the isomorphism $\pi\colon H\rightarrow H,A\mapsto PAP^{-1}$ makes the diagram commute.\bigskip

Assume that $H$ is not of type $\A$. Then $H=\{g\in\GL_n\mid g^TBg=B\}$ for some matrix $B\in\GL_n$. Let $\beta_1$ and $\beta_2$ be the $G$-invariant bilinear forms on $K^n$ defined by $Q^{-T}BQ^{-1}$ and $R^{-T}BR^{-1}$. By the previous lemma, there exists a $G$-equivariant automorphism $\varphi$ of $K^n$ such that 
$$
\beta_2(\varphi(v),\varphi(w))=\beta_1(v,w)
$$
for all $v,w\in K^n$. Let $S$ be the matrix corresponding to $\varphi$. Then 
$$
S^TQ^{-T}BQ^{-1}S=R^{-T}BR^{-1}
$$
and
$$
S\Diag(A,\dots,A,A^{-T},\dots,A^{-T},I_z)=\Diag(A,\dots,A,A^{-T},\dots,A^{-T},I_z)S
$$
for all $A\in G$. Take $P=R^{-1}S^{-1}Q$. Then $P^{-1}\in H$ and therefore $P\in H$. The isomorphism $\pi\colon H\rightarrow H,A\mapsto PAP^{-1}$ makes the diagram commute.
\end{proof}

\begin{prop}
For every $i\in\NN$, let $\iota_i'\colon G_i\hookrightarrow G_{i+1}$ be a diagonal embedding with the same signature $(l_i,r_i,z_i)$ as $\iota_i$. If the type of $G_i$ is $\B$, $\C$ or $\D$ for any $i\in\NN$, assume that $K=\overline{K}$. If the type of $G_i$ is $\B$ or $\D$ for any $i\in\NN$, assume in addition that $\cha(K)\neq2$. Then there exist isomorphisms $\varphi_i\colon G_i\rightarrow G_i$ making the diagram
$$\begin{tikzcd}
G_1 \arrow[hook,"\iota_1"]{r}\arrow["\id"]{d} & G_2 \arrow[hook,"\iota_2"]{r}\arrow["\varphi_2"]{d} & G_3 \arrow[hook,"\iota_3"]{r}\arrow["\varphi_3"]{d} & \dots\\
G_1 \arrow[hook,"\iota_1'"]{r} & G_2 \arrow[hook,"\iota_2'"]{r} & G_3 \arrow[hook,"\iota_3'"]{r} & \dots
\end{tikzcd}$$
commute. 
\end{prop}
\begin{proof}
We construct the isomorphisms $\varphi_i$ recursively in such a way that the $\varphi_i$ are also diagonal embeddings with signature $(1,0,0)$. Write $\varphi_1=\id$, let $i\geq 2$ and assume that $\varphi_{i-1}$ has has already been constructed. Then $\iota'_{i-1}\circ \varphi_{i-1}$ has the same signature as $\iota_{i-1}$. So by the previous lemma, there exists an isomorphism $\varphi_i$ making the diagram
$$\begin{tikzcd}
  G_{i-1}  \arrow[hook,"\iota_{i-1}"]{rr} \arrow["\id"]{d} && G_i \arrow["\varphi_i"]{d} \\
  G_{i-1} \arrow[hook,"\iota'_{i-1}\circ \varphi_{i-1}"]{rr} && G_i
\end{tikzcd}$$
commute that also has signature $(1,0,0)$ as a diagonal embedding.
\end{proof}

Recall that, when we replace
$$\begin{tikzcd}
G_1 \arrow[hook,"\iota_1"]{r} & G_2 \arrow[hook,"\iota_2"]{r} & G_3 \arrow[hook,"\iota_3"]{r} & \dots
\end{tikzcd}$$
by supersequences or infinite subsequences, we do not change $G$ or $V$. Therefore we may assume that each group $G_i$ has the same type and we will prove the Main Theorem for sequences of groups of type $\A$, $\B$, $\C$ and $\D$ separately. The proposition tells us that, if we replace $K$ by its algebraic closure, the limits $G$ and $V$ only depend on the signatures of the diagonal embeddings. Since $G$-Noetherianity of $V$ over $\overline{K}$ implies $G$-Noetherianity of $V$ over the original field $K$, we only have to consider one diagonal embedding per possible signature.

\subsection{Identifying $V$ with the inverse limit of a sequence of quotients/subspaces of matrix spaces}

We encounter the following Lie algebras:
$$\begin{array}{clclc}
\A_{n-1}:&\sl_n&=&\{P\in\gl_n\mid\tr(P)=0\}&\mbox{for }n\in\NN\\~\\
\B_n:&\o_{2n+1}&=&\left\{P\in\gl_{2n+1}~\middle|~P\begin{pmatrix}&&I_n\\&1\\I_n\end{pmatrix}+\begin{pmatrix}&&I_n\\&1\\I_n\end{pmatrix}P^T=0\right\}&\mbox{for }n\in\NN\\~\\
\C_n:&\sp_{2n}&=&\left\{P\in\gl_{2n}~\middle|~P\begin{pmatrix}&I_n\\-I_n\end{pmatrix}+\begin{pmatrix}&I_n\\-I_n\end{pmatrix}P^T=0\right\}&\mbox{for }n\in\NN\\~\\
\D_n:&\o_{2n}&=&\left\{P\in\gl_{2n}~\middle|~P\begin{pmatrix}&I_n\\I_n\end{pmatrix}+\begin{pmatrix}&I_n\\I_n\end{pmatrix}P^T=0\right\}&\mbox{for }n\in\NN
\end{array}$$
These are all subspaces of $\gl_m$ for some $m\in\NN$. Consider the symmetric bilinear form $\gl_m\times\gl_m\rightarrow K,(P,Q)\mapsto\tr(PQ)$. This map is non-degenerate and therefore the map $\gl_m\rightarrow\gl_m^*,P\mapsto(Q\mapsto\tr(PQ))$ is an isomorphism. By composing this map with the restriction map $\gl_m^*\rightarrow\sl_m^*$ and factoring out the kernel, we find that
\begin{eqnarray*}
\gl_m/\span(I_m)&\rightarrow&\sl_m^*\\
P\mod I_m&\mapsto&(Q\mapsto\tr(PQ))
\end{eqnarray*}
is an isomorphism. When $\cha(K)\neq2$ and $\g\subseteq\gl_m$ is a Lie algebra of type $\B$, $\C$ or $\D$, the restriction of the bilinear map to $\g\times\g$ is non-degenerate. So the map
\begin{eqnarray*}
\g&\rightarrow&\g^*\\
P&\mapsto&(Q\mapsto\tr(PQ))
\end{eqnarray*}
is an isomorphism. Since the map $\gl_n\rightarrow\gl_n^*$ is in fact $\GL_n$-equivariant, the maps $\gl_m/\span(I_m)\rightarrow\sl_m^*$ and $\g\rightarrow\g^*$ are all isomorphisms of representations of the groups acting on them. Using these isomorphisms, we identify the duals $\g_i^*$ of the Lie algebras of the groups $G_i$ with quotients/subspaces of spaces of matrices. This in particular allows us to define the coordinate rings of the $\g_i^*$ in terms of entries of matrices. For type $\A$, we get
$$
K[\gl_n/\span(I_n)]=\{f\in K[\gl_n]\mid \forall P\in\gl_n\forall\lambda\in K\colon f(P+\lambda I_n)=f(P)\}
$$
which is the graded subring 
$$
K[p_{k\ell}\mid k\neq\ell]\otimes_KK[p_{11}-p_{kk}\mid k\neq1]
$$
of $K[\gl_n]=K[p_{k\ell}\mid1\leq k,\ell\leq n]$. For type $\B$, assuming that $\cha(K)\neq 2$, we have
$$
\o_{2n+1}=\left\{\begin{pmatrix}P&v&Q\\-w^T&0&-v^T\\R&w&-P^T\end{pmatrix}\in\gl_{2n+1}~\middle|~\begin{array}{c}Q+Q^T=0\\R+R^T=0\end{array} \right\}
$$
and therefore we get
$$
K[\o_{2n+1}]=K[p_{k\ell},q_{k\ell},r_{k\ell},v_k,w_k\mid 1\leq k,\ell\leq n]/(q_{k\ell}+q_{\ell k},r_{k\ell}+r_{\ell k}).
$$
For type $\C$, we have 
$$
\sp_{2n}=\left\{\begin{pmatrix}P&Q\\R&-P^T\end{pmatrix}\in\gl_{2n}~\middle|~\begin{array}{c}Q=Q^T\\R=R^T\end{array} \right\}
$$
and we get
$$
K[\sp_{2n}]=K[p_{k\ell},q_{k\ell},r_{k\ell}\mid 1\leq k,\ell\leq n]/(q_{k\ell}-q_{\ell k},r_{k\ell}-r_{\ell k}).
$$
For type $\D$, assuming that $\cha(K)\neq 2$, we have
$$
\o_{2n}=\left\{\begin{pmatrix}P&Q\\R&-P^T\end{pmatrix}\in\gl_{2n}~\middle|~\begin{array}{c}Q+Q^T=0\\R+R^T=0\end{array} \right\}
$$
and get
$$
K[\o_{2n}]=K[p_{k\ell},q_{k\ell},r_{k\ell}\mid 1\leq k,\ell\leq n]/(q_{k\ell}+q_{\ell k},r_{k\ell}+r_{\ell k}).
$$
For Lie algebras $\g\subseteq\gl_m$ of type $\B$, $\C$ or $\D$, we will denote elements of $K[\g]$ by their representatives in $K[\gl_m]$. Define a grading on each of these coordinate rings by $\grad(r_{k\ell})=\grad(w_k)=0$, $\grad(p_{k\ell})=\grad(v_k)=1$ and $\grad(q_{k\ell})=2$ for all $k,\ell\in[n]$.

\subsection{Moving equations around}

Let $X\subsetneq V$ be a $G$-stable closed subset. For each $i\in\NN$, let $V_i$ be the vector space (we identified with) $\g_i^*$ which is acted on by $G_i$ by conjugation and let $X_i$ be the closure of the projection from $X$ to $V_i$. Then $X_i$ is a $G_i$-stable closed subset of $V_i$ for all $i\in\NN$ and there exists an $i\in\NN$ such that $X_i\neq V_i$. This means that the ideal $I(X_i)\subseteq K[V_i]$ is non-zero. Let $f$ be a non-zero element of $I(X_i)$ and let $d$ be its degree. The first step of the proof of the Main theorem is to use this polynomial $f$ to get elements $f_j$ of $I(X_j)$  such that $f_j\neq 0$, such that $\deg(f_j)\leq d$ and such that $f_j$ is ``off-diagonal" for all $j\gg i$. When the groups $G_i$ are of type $\B$, $\C$ or $\D$, this last condition means that $f_j$ is a polynomial in only the variables $r_{k\ell}$ and $w_k$. When the groups $G_i$ are of type $\A$, we similarly require that the $f_j$ are polynomials in the variables $p_{k\ell}$ with $k\in\mathscr{K}$ and $\ell\in\mathscr{L}$ for some disjoint sets $\mathscr{K},\mathscr{L}$. \bigskip

The projection maps $\pr_i\colon V_{i+1}\rightarrow V_i$ induce maps $\pr_i^*\colon K[V_i]\rightarrow K[V_{i+1}]$ which are injective and degree-preserving. We will see that, for many of the maps $\pr_i$ we will encounter, the map $\pr_i^*$ is also $\grad$-preserving. Since $X_{i+1}$ projects into $X_i$, we have $\pr_i^*(I(X_i))\subseteq I(X_{i+1})$. So $f$ induces non-zero elements $g_j\in I(X_j)$ of degree $d$ for all $j>i$.\bigskip

Let $A\colon K^k\rightarrow G_j$ be a polynomial map such that the map
\begin{eqnarray*}
K^k&\rightarrow&G_j\\
\Lambda&\mapsto&A(\Lambda)^{-1}
\end{eqnarray*}
is polynomial as well. Then $A(\Lambda)\cdot g_j\in I(X_j)$ for all $\Lambda\in K^k$ and therefore linear combinations of such elements also lie in $I(X_j)$. Note that we can view $A(\Lambda)\cdot g_j$ as a polynomial in the entries of $\Lambda$ whose coefficients are elements of $K[V_j]$. Let $R$ be a $K$-algebra and $h\in R[x]$ a polynomial. Then, since the field $K$ is infinite, one sees using a Vandermonde matrix that the coefficients of $h$ are contained in the $K$-span of $\{h(\lambda)\mid \lambda\in K\}$. Applying this fact $k$ times, we see that all the coeffiecients of $A(\Lambda)\cdot g_j$ lie in $\span(A(\Lambda)\cdot g_j\mid\Lambda\in K^k)\subseteq I(X_j)$.\bigskip

We will let $f_j$ be a certain one of these coefficients. We have $\deg(f_j)\leq d$ by construction and we will choose $A$ in such a way that $f_j$ is ``off-diagonal". We will see that $f_j$ is obtained from $g_j$ by substituting variables into the top-graded part of $g_j$ with respect to the right grading (in most cases $\deg$ or $\grad$). Since the polynomial $g_j$ is non-zero, so is its top-graded part with respect to any grading. So it then suffices to check that this top-graded part does not become zero after the substitution. In the cases where is this not obvious, it will follow from a lemma stating that a certain morphism is dominant.

\subsection{Using knowledge about stable closed subsets of the \texorpdfstring{``off-diagonal"}{"off-diagonal"} part}
The space $V_j$ consists of matrices. When we have an ``off-diagonal" polynomial which is contained in $I(X_j)$, we know that the projection $Y$ of $X_j$ onto some off-diagonal submatrix cannot form a dense subset of the projection $W$ of the whole of $V_j$. We then give $W$ the structure of a representation such that $Y$ is stable and use the fact the we know that the ideal of $Y$ contains a non-zero polynomial of degree at most $d$ to find conditions that hold for all elements of $Y$. These in turn give conditions that must hold for all elements of $X_j$, which will be enough to prove that $X$ is $G$-Noetherian. 

\section{Limits of classical groups of type A}

In this section, we let $G$ be the direct limit of a sequence 
$$\begin{tikzcd}
\SL_{n_1} \arrow[hook,"\iota_1"]{r} & \SL_{n_2} \arrow[hook,"\iota_2"]{r} & \SL_{n_3} \arrow[hook,"\iota_3"]{r} & \dots
\end{tikzcd}$$
of diagonal embeddings given by
\begin{eqnarray*}
\iota_i\colon \SL_{n_i}&\hookrightarrow&\SL_{n_{i+1}}\\
A&\mapsto&\Diag(A,\dots,A,A^{-T},\dots,A^{-T},I_{z_i})
\end{eqnarray*}
with $l_i$ blocks $A$ and $r_i$ blocks $A^{-T}$ for some $l_i\in\NN$ and $r_i,z_i\in\ZZ_{\geq0}$ with $l_i\geq r_i$. We let $V$ be the inverse limit of the sequence
$$\begin{tikzcd}
\gl_{n_1}/\span(I_{n_1}) & \arrow[two heads]{l} \gl_{n_2}/\span(I_{n_2}) & \arrow[two heads]{l} \gl_{n_3}/\span(I_{n_3}) & \arrow[two heads]{l} \dots
\end{tikzcd}$$
where the maps are given by
\begin{eqnarray*}
\gl_{n_{i+1}}/\span(I_{n_{i+1}})&\twoheadrightarrow&\gl_{n_i}/\span(I_{n_i})\\
\begin{pmatrix}P_{11}&\dots&P_{1l_i}&\bullet&\dots&\bullet&\bullet\\\vdots&&\vdots&\vdots&&\vdots&\vdots\\P_{l_i1}&\dots&P_{l_il_i}&\bullet&\dots&\bullet&\bullet\\\bullet&\dots&\bullet&Q_{11}&\dots&Q_{1r_i}&\bullet\\\vdots&&\vdots&\vdots&&\vdots&\vdots\\\bullet&\dots&\bullet&Q_{r_i1}&\dots&Q_{r_ir_i}&\bullet\\\bullet&\dots&\bullet&\bullet&\dots&\bullet&\bullet\end{pmatrix}\mod I_{n_{i+1}}&\mapsto&\sum_{k=1}^{l_i}P_{kk}-\sum_{\ell=1}^{r_i}Q_{\ell\ell}^{T}\mod I_{n_i}.
\end{eqnarray*}
Here each $\bullet$ represents some matrix of the appropriate size. Our goal is to prove that the inverse limit $V$ of this sequence is $G$-Noetherian.\bigskip

Take $\alpha=\#\{i\mid l_i>1\}$, $\beta=\#\{i\mid r_i>0\}$, $\gamma=\#\{i\mid z_i>0\}\in\ZZ_{\geq0}\cup\{\infty\}$. Then we have $\alpha+\beta+\gamma=\infty$, since $G$ is assumed to be infinite-dimensional. Based on $\alpha,\beta,\gamma$ we distinguish the following cases:
\begin{itemize}
\item[(1)] $\alpha+\beta<\infty$;
\item[(2)] $\alpha+\beta=\gamma=\infty$;
\item[(3a)] $\beta=\infty$, $\gamma<\infty$ and $\cha(K)\neq 2$ or $2\nmid n_i$ for all $i\gg0$;
\item[(3b)] $\beta=\infty$, $\gamma<\infty$, $\cha(K)=2$ and $2\mid n_i$ for all $i\gg0$;
\item[(4a)] $\beta+\gamma<\infty$ and $\cha(K)\nmid n_i$ for all $i\gg0$; and
\item[(4b)] $\beta+\gamma<\infty$ and $\cha(K)\mid n_i$ for all $i\gg0$.
\end{itemize}
Note here that if $\gamma<\infty$, then $n_i|n_{i+1}$ for all $i\gg0$. Denote the element of $V$ representated by the sequence of zero matrices by $0$. 

\begin{thm}\label{thm_mainsl}
The space $V$ is $G$-Noetherian. Any $G$-stable closed subset of $V$ is a finite union of irreducible $G$-stable closed subsets. The irreducible $G$-stable closed subsets of $V$ are $\{0\}$ and $V$ together with 
$$
\left\{(P_i\mod I_{n_i})_i\in V~\middle|~\forall i\gg0\colon\rk(P_i,I_{n_i})\leq k\right\}
$$
for $k\in\NN$ in case (1) and together with 
$$
\left\{(P_i\mod I_{n_i})_i\in V~\middle|~\forall i\gg0\colon\tr(P_i)=\mu\right\}
$$
for $\mu\in K$ in cases (3b) and (4b).
\end{thm}

Here we call a closed subset $X\subseteq V$ irreducible when the following condition holds: if $X=Y\cup Z$ for (not necessarily $G$-stable) closed subsets $Y,Z\subseteq X$, then $X=Y$ or $X=Z$. The following proposition expresses the irreduciblility of a closed subset of $V$ in terms of the closures of its projections.

\begin{prop}\label{prop_irr2irrinf}
Let
$$\begin{tikzcd}
W_1& \arrow[two heads]{l} W_2 & \arrow[two heads]{l} W_{3} & \arrow[two heads]{l} \dots
\end{tikzcd}$$
be a sequence of finite-dimensional vector spaces with inverse limit $W$. Let $X\subseteq W$ be a closed subset and let $X_i$ be the closure of the projection of $X$ to $W_i$. Then the following are equivalent:
\begin{itemize}
\item[(1)] $X$ is irreducible.
\item[(2)] $X_i$ is irreducible for all $i\geq 1$. 
\item[(3)] $X_i$ is irreducible for all $i\gg0$. 
\end{itemize} 
\end{prop}
\begin{proof}
Suppose that $X_i$ is reducible for some $i\in\NN$. Then $X_i=Y\cup Z$ for some closed subsets $Y,Z\subsetneq X_i$. In this case, we see that
$$
X=(\pr_i^{-1}(Y)\cap X)\cup(\pr_i^{-1}(Z)\cap X),\quad \pr_i^{-1}(Y)\cap X,\pr_i^{-1}(Z)\cap X\subsetneq X
$$ 
and so $X$ is reducible. This establishes $(1)\Rightarrow(2)$. The implication $(2)\Rightarrow(3)$ is trivial. So next, if $X=Y\cup Z$ for some closed subsets $Y,Z\subsetneq X$ with closures $Y_i,Z_i$ in $W_i$, then $X_i=Y_i\cup Z_i$ for all $i\in\NN$ and $Y_i,Z_i\subsetneq X_i$ for all $i\gg 0$. So in this case, we see that $X_i$ is reducible for $i\gg0$.
\end{proof}

\subsection{The case \texorpdfstring{$\alpha+\beta<\infty$}{alpha+beta<oo}}

By replacing
$$\begin{tikzcd}
\SL_{n_1} \arrow[hook,"\iota_1"]{r} & \SL_{n_2} \arrow[hook,"\iota_2"]{r} & \SL_{n_3} \arrow[hook,"\iota_3"]{r} & \dots
\end{tikzcd}$$
with some infinite subsequence, we may assume that $(l_i,r_i)=(1,0)$ and $z_i>0$ for all $i\in\NN$. Then, by replacing the sequence by a supersequence, we may assume that $n_i=i$ and $z_i=1$ for all $i\in\NN$. So we consider the inverse limit $V=\gl_\infty/\span(I_\infty)$ of the sequence
$$\begin{tikzcd}
\gl_{1}/\span(I_{1}) & \arrow[two heads]{l} \gl_{2}/\span(I_{2}) & \arrow[two heads]{l} \gl_{3}/\span(I_{3}) & \arrow[two heads]{l} \dots
\end{tikzcd}$$
acted on by the group $G=\SL_\infty$. The $\SL_\infty$-stable closed subsets of $\gl_\infty/\span(I_\infty)$ correspond one-to-one to the $\SL_\infty$-stable closed subsets $X$ of $\gl_\infty$ such that 
$$
X+\span(I_{\infty})=X.
$$
Theorem \ref{thm_maingl} therefore tells us exactly what the $G$-stable closed subsets of $V$ are. The next proposition shows that Theorem \ref{thm_maingl} implies case (1) of Theorem \ref{thm_mainsl} . 

\begin{prop}\label{prop_tuplerank=suptuplerank}
Let $P_1,\dots,P_k$ be elements of $\gl_\infty$. Then we have
$$
\rk(P_1,\dots,P_k)=\sup\{\rk(\pr_n(P_1),\dots,\pr_n(P_k))\mid n\in\NN\}.
$$
\end{prop}
\begin{proof}
We have $\rk(\pr_n(P_1),\dots,\pr_n(P_k))\leq \rk(\mu_1P_1+\dots+\mu_kP_k)$ for all $n\in\NN$ and $(\mu_1:\dots:\mu_k)\in\PP^{k-1}$. So 
$$
r:=\sup\{\rk(\pr_n(P_1),\dots,\pr_n(P_k))\mid n\in\NN\}\leq \rk(P_1,\dots,P_k)
$$
with equality when $r=\infty$. Suppose that $r<\infty$ and consider the descending chain 
$$
Y_1\supseteq Y_2\supseteq Y_3\supseteq Y_4\supseteq\dots
$$
of closed subsets of $\PP^{k-1}$ defined by 
$$
Y_n=\left\{(\mu_1:\dots:\mu_k)\in\PP^{k-1}~\middle|~\rk(\mu_1\pr_n(P_1)+\dots+\mu_k\pr_n(P_k))\leq r\right\}.
$$
By construction, each $Y_n$ is non-empty. And by the Noetherianity of $\PP^{k-1}$, the chain stabilizes. Let $(\mu_1:\dots:\mu_k)\in\PP^{k-1}$ be an element contained in $Y_n$ for all $n\in\NN$. Then we see that $\rk(P_1,\dots,P_k)\leq \rk(\mu_1P_1+\dots+\mu_kP_k)\leq r$.
\end{proof}

So we proceed to prove Theorem \ref{thm_maingl}. The following proposition, which is due to Jan Draisma, connects the tuple rank of a matrix $P$ with the identity matrix to the rank of off-diagonal submatrices of matrices similar to $P$.

\begin{prop}\label{prop_tuplerank}
Let $k,m,n\in\ZZ_{\geq0}$ be such that $n\geq 2m\geq 2(k+1)$, let $\mathscr{K},\mathscr{L}$ be disjoint subsets of $[n]$ of size $m$ and let $P$ be an $n\times n$ matrix. Then $\rk(P,I_n)\leq k$ if and only if the submatrix $Q_{\mathscr{K},\mathscr{L}}$ of $Q$ has rank at most $k$ for every $Q\sim P$.
\end{prop}
\begin{proof}
Suppose that $\rk(P,I_n)\leq k$. Let $Q\sim P$ be a similar matrix. Then $\rk(Q,I_n)\leq k$. So since $\mathscr{K}\cap \mathscr{L}=\emptyset$ and the off-diagonal entries of $Q$ and $Q-\lambda I_n$ are equal for all $\lambda\in K$, we see that $\rk(Q_{\mathscr{K},\mathscr{L}})\leq k$.\bigskip

Suppose that the submatrix $Q_{\mathscr{K},\mathscr{L}}$ has rank at most $k$ for every $Q\sim P$. Then this statement still holds when we replace $\mathscr{K}$ and $\mathscr{L}$ by subsets of themselves of size $k+1$. This reduces the proposition to the case $m=k+1$. Now the statement we want to prove is implied by the following coordinate-free version:
\begin{itemize}
\item[(*)] Let $V$ be a vector space of dimension $n$ and let $\varphi\colon V\rightarrow V$ be an endomorphism. If the induced map $\varphi\colon W\rightarrow V/W$ has a non-trivial kernel for all $(k+1)$-dimensional subspaces $W$ of $V$, then $\varphi$ has an eigenvalue of geometric multiplicity at least $n-k$.
\end{itemize}
Indeed, taking $\varphi\colon K^n\to K^n$ the endomorphism corresponding to $P$ and $W\subseteq K^n$ a $(k+1)$-dimensional subspace, we can first replace $P$ be a matrix $Q\sim P$ to get $W=K^{k+1}\times\{0\}$. Since $Q$ is similar to all its conjugates by permutation matrices, we know that $\det(Q_{\mathscr{K},\mathscr{L}})=0$ for all disjoint subsets of $\mathscr{K},\mathscr{L}\subseteq[n]$ of size $m$. Hence $Q_{[n]\setminus[k+1],[k+1]}$ has rank at most $k$. So the induced map $W\to V/W$ has a non-trivial kernel. We conclude from (*) that 
$$
\rk(P-\lambda I_n)=\rk(Q-\lambda I_n)\leq n-(n-k)=k
$$
for some $\lambda\in K$. So $\rk(P,I_n)\leq k$.\bigskip

To prove (*), consider the incidence variety
$$
Z=\left\{(W,[v])\in\Gr_{k+1}(V)\times\PP(V)\mid v,\varphi(v)\in W\right\}
$$
and let $\pi_1,\pi_2$ be the projections from $Z$ to the Grassmannian $\Gr_k(V)$ and to $\PP(V)$. By assumption $\pi_1$ is surjective. So we have 
$$
\dim Z \geq \dim(\Gr_{k+1}(V))=(k+1)(n-k-1).
$$
On the other hand, let $v\in V\setminus\{0\}$ be a non-eigenvector of $\varphi$. Then $\pi_1(\pi_2^{-1}([v]))$ consists of all $W\in\Gr_{k+1}(V)$ containing $\span(v,\varphi(v))$ and these form the Grassmannian $\Gr_{k-1}(V/\span(v,\varphi(v)))$ of dimension $(k-1)(n-k-1)$. Thus the union of the fibres $\pi_2^{-1}([v])$ for $v$ not an eigenvector of $\varphi$ has dimension at most 
$$
(k-1)(n-k-1)+\dim(\PP(V))=(k+1)(n-k-1)+2k+1-n.
$$
This dimension is strictly smaller than $\dim(Z)$. Let $v$ be an eigenvector of $\varphi$. Then $\pi_1(\pi_2^{-1}([v]))$ consists of all $W\in\Gr_{k+1}(V)$ with $v\in W$ and these form the Grassmannian $\Gr_k(V/\span(v))$ of dimension $k(n-k-1)$. So we see that the union of the eigenspaces of $\varphi$ must have dimension at least $\dim(Z)-k(n-k-1)+1\geq n-k$. Hence some eigenspace of $\varphi$ must have dimension al least $n-k$.
\end{proof}

\begin{de}
For $n\in\NN$, we call a polynomial $f\in K[\gl_n]$ off-diagonal if 
$$
f\in K[p_{k\ell}\mid k\in\mathscr{K},\ell\in\mathscr{L}]
$$
for some disjoint subsets $\mathscr{K},\mathscr{L}\subset[n]$ of size $m\leq (n-1)/2$. 
\end{de}

\begin{lm}\label{lm_lowranklowtuplerank}
Let $n\in\NN$ be an integer, let $Y$ be an $\SL_n$-stable closed subset of $\gl_n$ and suppose that $I(Y)$ contains a non-zero off-diagonal polynomial $f$. Then $\rk(P,I_n)<\deg(f)$ for all $P\in Y$.
\end{lm}
\begin{proof}
Let $\mathscr{K},\mathscr{L}\subset[n]$ be disjoint subsets of size $m\leq n/2$ and let 
$$
f\in K[p_{k\ell}\mid k\in\mathscr{K},\ell\in\mathscr{L}]\cap I(Y)
$$
be a non-zero element. If $m=0$, then $f$ is constant and $Y=\emptyset$. So in particular, $\rk(P,I_n)<\deg(f)$ for all $P\in Y$. For $m>0$, let $Z$ be the closure of the set 
$$
\{(y_{k\ell})_{k\in\mathscr{K},\ell\in\mathscr{L}}\mid(y_{k\ell})_{k,\ell}\in Y\}
$$
in $\gl_m$. Then $f\in I(Z)$. By conjugating with with $\pm1$ times a permutation matrix, we may assume that $\mathscr{K}=[m]$ and $\mathscr{L}=[2m]\setminus[m]$. Now consider the map
\begin{eqnarray*}
\GL_m\times\GL_m&\to&\SL_n\\
(A,B)&\mapsto&\Diag(A,B,I_{n-2m-1},\det(AB)^{-1}).
\end{eqnarray*}
Since $Y$ is $\GL_m\times\GL_m$-stable, we see that $Z$ is closed under $\GL_m\times\GL_m$ acting by left and right multiplication. So $Z$ must consist of all matrices of rank at most $\ell$ for some $\ell\leq m$. Since $f\in I(Z)$, we see that $\ell<\min(m,\deg(f))$. So by Proposition \ref{prop_tuplerank}, we see that $Y$ consists of matrices $P$ such that $\rk(P,I_n)<\min(m,\deg(f))\leq\deg(f)$.
\end{proof}

\begin{re}\label{re_lowranklowtuplerank}
Let $Y$ be a $\SL_n$-stable closed subset of $\gl_n/\span(I_n)$. Then we can apply Lemma \ref{lm_lowranklowtuplerank} to $Y$ by considering its inverse image in $\gl_n$. So if $I(Y)$ contains a non-zero off-diagonal polynomial $f$, then $\rk(P,I_n)<\deg(f)$ for all $(P\mod I_n)\in Y$.
\end{re}

Let $X$ be a proper $\SL_\infty$-stable closed subset of $\gl_\infty$. Denote the closure of the projection of $X$ to $\gl_n$ by $X_n$ and let $I(X_n)\subseteq K[\gl_n]$ be its corresponding ideal.

\begin{lm}\label{lm_boundedtrk}
Let $m$ be a positive integer and suppose that $I(X_m)$ contains a non-zero polynomial $f$. Then $rk(P,I_\infty)<\deg(f)$ for all $P\in X$.
\end{lm}
\begin{proof}
Note that the morphism $X_n\rightarrow X_m$ is dominant for all positive integers $m\leq n$. So it suffices to prove that $\rk(\pr_n(P),I_n)<\deg(f)$ for $n\gg0$. Let $n\geq 2m+1$ be an integer. Then $f$ induces the element
$$
g=\left(\begin{pmatrix}P&Q&\bullet\\R&S&\bullet\\\bullet&\bullet&\bullet\end{pmatrix}\mapsto f(P)\right)
$$
of $I(X_n)$ where $P,Q,R,S\in\gl_m$. This allows us to assume that $\deg(f)<m$ without loss of generality. For $\lambda\in K$, consider the matrix
$$
A(\lambda)=\begin{pmatrix}I_m&\lambda I_m\\&I_m\\&&I_{n-2m}\end{pmatrix}\in\SL_n.
$$
We have
$$
A(\lambda)\begin{pmatrix}P&Q&\bullet\\R&S&\bullet\\\bullet&\bullet&\bullet\end{pmatrix}A(\lambda)^{-1}=\begin{pmatrix}P+\lambda R&Q+\lambda (S-P)-\lambda^2R&\bullet\\R&S-\lambda R&\bullet\\\bullet&\bullet&\bullet\end{pmatrix}
$$
for all $\lambda\in K$. So we see that if we let $A(\lambda)$ act on $g$, we obtain the element
$$
h_\lambda=\left(\begin{pmatrix}P&Q&\bullet\\R&S&\bullet\\\bullet&\bullet&\bullet\end{pmatrix}\mapsto f(P+\lambda R)\right)
$$
of $I(X_n)$. Let $k+1$ be the degree of $f$ and let $f_{k+1}$ be the homogeneous part of $f$ of degree $k+1$. Then the homogeneous part of $h_\lambda$ of degree $k+1$ in $\lambda$ equals the polynomial $\lambda^{k+1}f_{k+1}(R)$. Since the field $K$ is infinite, the polynomial $f_{k+1}(R)$ is a linear combination of the $h_\lambda$. Hence $f_{k+1}(R)\in I(X_n)$. So $\rk(P,I_n)<\deg(f)$ for all $P\in X_n$ by Lemma \ref{lm_lowranklowtuplerank} and therefore $\rk(P,I_\infty)<\deg(f)$ for all $P\in X$.
\end{proof}

\begin{lm}\label{lm_randomtopleft}
Let $k<n$ be non-negative integers and let $P\in\gl_{2n}$ and $Q\in\gl_n$ be matrices with $\rk(P)=k$ and $\rk(Q)\leq k$. Then $P$ is similar to
$$
\begin{pmatrix}Q&Q_{12}\\Q_{21}&Q_{22}\end{pmatrix}
$$
for some $Q_{12},Q_{21},Q_{22}\in\gl_n$.
\end{lm}
\begin{proof}
First note that $\rk(P,I_{2n})=2n-\dim\ker(P)=k$, since $0$ has the highest geometric multiplicity among all eigenvalues of $P$. Since $2(k+1)\leq 2n$, it follows by Proposition \ref{prop_tuplerank} that
$$
P\sim\begin{pmatrix}\bullet&\bullet\\R&\bullet\end{pmatrix}
$$
for some matrix $R\in\gl_n$ with $\rk(R)=k$. By conjugating the latter matrix with $\Diag(g,I_n)$ for some $g\in\GL_n$ such that $g\ker(R)\subseteq\ker(Q)$, we see that
$$
\begin{pmatrix}\bullet&\bullet\\R&\bullet\end{pmatrix}\sim\begin{pmatrix}\bullet&\bullet\\R'&\bullet\end{pmatrix}
$$
for some matrix $R'\in\gl_n$ with $\rk(R')=k$ and $\ker(R')\subseteq\ker(Q)$. This means that $Q=SR'$ for some $S\in\gl_n$. Since both $R'$ and any matrix similar to $P$ have rank $k$, we see that the matrix on the right must be of the form
$$
\begin{pmatrix}TR'&\bullet\\R'&\bullet\end{pmatrix}
$$
for some $T\in\gl_n$. Now note that the matrix
$$
\begin{pmatrix}I_n&S-T\\0&I_n\end{pmatrix}\begin{pmatrix}TR'&\bullet\\R'&\bullet\end{pmatrix}\begin{pmatrix}I_n&T-S\\0&I_n\end{pmatrix}=\begin{pmatrix}SR'&\bullet\\R'&\bullet\end{pmatrix}=\begin{pmatrix}Q&\bullet\\R'&\bullet\end{pmatrix}
$$
is similar to $P$ and of the form we want.
\end{proof}

\begin{prop}\label{prop_pointclosure}
Let $P\in\gl_\infty$ be an element. Then either the orbit of $P$ is dense in $\gl_\infty$ or $k=\rk(P-\lambda I_\infty)<\infty$ for some unique $\lambda\in K$. In the second case, the closure of the orbit of $P$ equals the irreducible closed subset $\{Q\in\gl_\infty\mid\rk(Q-\lambda I_\infty)\leq k\}$ of $\gl_{\infty}$.
\end{prop}
\begin{proof}
Let $X$ be the closure of the orbit of $P$. Then either $X=\gl_\infty$ or $\rk(P,I_\infty)=k$ for some $k\in\ZZ_{\geq0}$ by Lemma \ref{lm_boundedtrk}. In the second case, we see that $\rk(P-\lambda I_\infty)=k$ for some unique $\lambda\in K$ by Remark \ref{re_uniquehighgeomult}. Our goal is to prove that $X=\{Q\in\gl_\infty\mid\rk(Q-\lambda I_\infty)\leq k\}$. Using the $\SL_\infty$-equivariant affine isomorphism
\begin{eqnarray*}
\gl_\infty&\rightarrow&\gl_\infty\\
Q&\mapsto&Q-\lambda I_\infty
\end{eqnarray*}
we may assume that $\lambda=0$ and hence that $k=\rk(P)$ is finite. It suffices to prove that
$$
\pr_n(\{Q\in\gl_\infty\mid\rk(Q)\leq k\})=\{Q\in\gl_n\mid\rk(Q)\leq k\}=\pr_n(\SL_\infty\cdot P)
$$ 
for all $n\gg0$ since the middle set is irreducible. See Proposition \ref{prop_irr2irrinf}. The inclusions
$$
\pr_n(\SL_\infty\cdot P)\subseteq\pr_n(\{Q\in\gl_\infty\mid\rk(Q)\leq k\})\subseteq\{Q\in\gl_n\mid\rk(Q)\leq k\}
$$ 
are clear for all $n\in\NN$. Let $n>k$ be an integer such that the rank of $\pr_{2n}(P)$ equals~$k$. Then
$$
\{Q\in\gl_n\mid\rk(Q)\leq k\}\subseteq\pr_n(\SL_{2n}\cdot\pr_{2n}(P))\subseteq \pr_n(\SL_\infty\cdot P)
$$
by Lemma \ref{lm_randomtopleft}. So indeed $\pr_n(\{Q\in\gl_\infty\mid\rk(Q)\leq k\})=\pr_n(\SL_\infty\cdot P)$ for all $n\gg0$. 
\end{proof}

\begin{lm}\label{lm_containedinunion}
Let $m$ be a positive integer and suppose that $I(X_m)$ contains a non-zero polynomial $f$ with $\deg(f)<m$. Let $g(t)=f(tI_m)\in K[t]$ be the restriction of $f$ to $\span(I_m)$. Then $X$ is contained in 
$$
\bigcup_{\lambda}\left\{Q\in\gl_\infty~\middle|~\rk(Q-\lambda I_\infty)<\deg(f)\right\}
$$
where $\lambda\in K$ ranges over the zeros of $g$.
\end{lm}
\begin{proof}
Let $P$ be an element of $X$. Since $f$ is non-zero, we know that $X$ is a proper $\SL_\infty$-stable closed subset of $\gl_\infty$. Hence the orbit of $P$ cannot be dense in $\gl_\infty$. So $k=\rk(P-\lambda I_\infty)<\deg(f)$ for some $\lambda\in K$ by Lemma \ref{lm_boundedtrk}. This $\lambda$ is unique and the closure of the orbit of $P$ equals $\{Q\in\gl_\infty\mid\rk(Q-\lambda I_\infty)\leq k\}$ by Proposition~\ref{prop_pointclosure}. So we see that $\lambda I_\infty$ is an element of $X$. So $\lambda I_m$ is an element of $X_m$ and hence $g(\lambda)=f(\lambda I_m)=0$. We see that for all $P\in X$ there is a $\lambda\in K$ with $g(y)=0$ such that\medskip

\hfill$\displaystyle P\in \left\{Q\in\gl_\infty~\middle|~\rk(Q-\lambda I_\infty)<\deg(f)\right\}.$\smallskip
\end{proof}

\begin{prop}\label{prop_closedsubsetsnotcontaining}
Either the $\SL_\infty$-stable closed subset $\span(I_\infty)$ of $\gl_\infty$ is contained in $X$ or there exist $\lambda_1,\dots,\lambda_\ell\in K$ and $k_1,\dots,k_\ell\in\ZZ_{\geq0}$ such that 
$$
X=\bigcup_{i=1}^\ell\{Q\in\gl_\infty\mid\rk(Q-\lambda_i I_\infty)\leq k_i\}.
$$
\end{prop}
\begin{proof}
Assume that $\span(I_\infty)$ is not contained in $X$. Then, for some $m\in\NN$, $X_m$ is a proper subset of $\gl_m$ that does not contain $\span(I_m)$. The ideal $I(X_m)$ must contain a non-zero polynomial $f$ such that the polynomial $g(t)=f(tI_m)\in K[t]$ is non-zero. By Lemma \ref{lm_containedinunion}, we see that $X$ is contained in 
$$
\bigcup_{\lambda}\left\{Q\in\gl_\infty~\middle|~\rk(Q-\lambda I_\infty)<\deg(f)\right\}
$$
where $\lambda\in K$ ranges over the finitely many zeros of $g$. Take
$$
\Lambda=\left\{\lambda\in K~\middle|~g(\lambda)=0,\exists P\in X\colon \rk(P-\lambda I_\infty)<\deg(f)\right\}
$$ 
and take
$$
k_\lambda=\max\{\rk(P-\lambda I_\infty)\mid P\in X,\rk(P-\lambda I_\infty)<\infty\}
$$
for all $\lambda\in\Lambda$. Then we see that
$$
X=\bigcup_{\lambda\in\Lambda}\{Q\in\gl_\infty\mid\rk(Q-\lambda I_\infty)\leq k_\lambda\}
$$
using Proposition \ref{prop_pointclosure}.
\end{proof}

The proposition implies in particular that any descending chain of $\SL_\infty$-stable closed subsets of $\gl_\infty$ stablizes as long as one of these subsets does not contain $\span(I_\infty)$. Next we will classify the subsets that do contain $\span(I_\infty)$.  

\begin{prop}\label{prop_rkIleqk}
Let $k$ be a non-negative integer. Then the $\SL_\infty$-stable subset 
$$
\{P\in\gl_\infty\mid\rk(P,I_\infty)\leq k\}
$$
of $\gl_\infty$ is closed and irreducible.
\end{prop}
\begin{proof}
Using Proposition \ref{prop_tuplerank=suptuplerank}, we see that 
$$
\{P\in\gl_\infty\mid\rk(P,I_\infty)\leq k\}
$$
its the inverse limit of its projections $\{P\in\gl_n\mid\rk(P,I_n)\leq k\}$ onto $\gl_n$. So it suffices to show that this is a closed irreducible subset of $\gl_n$ for all $n\in\NN$. See Proposition~\ref{prop_irr2irrinf}.
The subset $\{P\in\gl_n\mid\rk(P,I_n)\leq k\}$ is the inverse image of the subset
$$
Y=\left\{(P,Q)\in\gl_n^2~\middle|~\rk(P,Q)\leq k\right\}
$$
under the map $\gl_n\rightarrow\gl_n^2,P\mapsto(P,I_n)$. The subset $Y$ is closed in $\gl_n^2$ since it is the image of the closed subset
$$
\left\{((\mu_1:\mu_2),P,Q)\in\PP^1\times\gl_n^2~\middle|~\rk(\mu_1P+\mu_2Q)\leq k\right\}
$$
under the projection map along the complete variety $\PP^1$. So $\{P\in\gl_n\mid\rk(P,I_n)\leq k\}$ is a closed subset of $\gl_n$. This subset is also the image of the map 
\begin{eqnarray*}
\{Q\in\gl_n\mid \rk(Q)\leq k\}\times K&\to&\gl_n\\
(Q,\lambda)&\mapsto&Q+\lambda I_n
\end{eqnarray*}
and hence irreducible.
\end{proof}

\begin{prop}\label{prop_closedsubsetcontaining}
Suppose that $X$ contains $\span(I_\infty)$. Then 
$$
X=\{P\in\gl_\infty\mid\rk(P,I_\infty)\leq k\}\cup Y
$$
for some non-negative integer $k$ and some $\SL_\infty$-stable closed subset $Y$ of $\gl_\infty$ that does not contain $\span(I_\infty)$.
\end{prop}
\begin{proof}
Since $X$ is a proper subset of $\gl_\infty$, we know that 
$$
X\subseteq \{P\in\gl_\infty\mid\rk(P,I_\infty)\leq \ell\}
$$
for some $\ell\in\ZZ_{\geq0}$ by Lemma \ref{lm_boundedtrk}. Let $k$ be the maximal non-negative integer such that
$$
\{P\in\gl_\infty\mid\rk(P,I_\infty)\leq k\}\subseteq X.
$$
We will prove the statement by induction on the difference between $\ell$ and $k$.\bigskip

Suppose that $\ell=k$. Then $X=\{P\in\gl_\infty\mid\rk(P,I_\infty)\leq k\}$ and the statement holds. Now suppose that $\ell>k$ and let $Y'$ be an $\SL_\infty$-stable closed subset of $\gl_\infty$ that does not contain $\span(I_\infty)$ such that
$$
X\cap\{P\in\gl_\infty\mid\rk(P,I_\infty)\leq \ell-1\}=\{P\in\gl_\infty\mid\rk(P,I_\infty)\leq k\}\cup Y'.
$$
Consider the set $Z=\{\lambda\in K\mid\exists P\in X\colon\rk(P-\lambda I_\infty)=\ell\}$ and fix an element $Q\in\gl_\infty$ with $\rk(Q)=\ell$. By Proposition \ref{prop_pointclosure}, we know for $\lambda\in K$ that $Q+\lambda I_\infty\in X$ if and only if $\lambda\in Z$. This shows that $Z$ is a closed subset of $K$. So either $Z=K$ or $Z$ is finite. If $Z=K$, then we see that $X$ contains all $P\in\gl_\infty$ with $\rk(P,I_\infty)\leq\ell$ by Proposition \ref{prop_pointclosure}. Since $\ell>k$, this is not true and hence $Z$ is finite. Take
$$
Y=Y'\cup \bigcup_{\lambda\in Z}\{P\in\gl_\infty\mid\rk(P-\lambda I_\infty)\leq \ell\}.
$$
Then we see that $X=\{P\in\gl_\infty\mid\rk(P,I_\infty)\leq k\}\cup Y$.
\end{proof}

\begin{proof}[Proof of Theorem \ref{thm_maingl}]
Let $S$ be the set pairs $(k,f)$ where $k\in\ZZ_{\geq-1}$ and where $f\colon K\rightarrow\ZZ_{\geq k}$ is a function such that $f^{-1}(\ZZ_{>k})$ is finite.  Define a partial ordering on $S$ by $(k,f)\leq (\ell,g)$ when $k\leq \ell$ and $f(\lambda)\leq g(\lambda)$ for all $\lambda\in K$. Then for all $(k,f)\in S$, the set $\{(k,g)\in S\mid (k,g)\leq (k,f)\}$ is finite. So any descending chain in $S$ stabilizes. For a proper $\SL_\infty$-stable closed subset $X$ of $\gl_\infty$, let $k_X$ be the maximal integer such that $\{P\in\gl_\infty\mid\rk(P,I_\infty)\leq k_X\}\subseteq X$ and let $f_X\colon K\rightarrow\ZZ_{\geq k}$ be the function sending $\lambda\in K$ to the maximal $k$ such that $\{P\in\gl_\infty\mid\rk(P-\lambda I_\infty)\leq k\}\subseteq X$. Then, by Propositions \ref{prop_closedsubsetsnotcontaining} and \ref{prop_closedsubsetcontaining}, we see that
$$
X=\{P\in\gl_\infty\mid\rk(P,I_\infty)\leq k_X\}\cup\bigcup_{\lambda\in f_X^{-1}(\ZZ_{>k_X})}\left\{P\in\gl_\infty~\middle|~\rk(P-\lambda I_\infty)\leq f_X(\lambda)\right\}
$$
and that the map $X\mapsto (k_X,f_X)$ is an order preserving bijection between the set of proper $\SL_\infty$-stable closed subsets of $\gl_\infty$ and $S$. Now consider a descending chain 
$$
X_1\supseteq X_2\supseteq X_3\supseteq X_4\supseteq\dots
$$
of $\SL_\infty$-stable closed subsets of $\gl_\infty$. We get a descending chain 
$$
(k_{X_1},f_{X_1})\geq (k_{X_2},f_{X_2})\geq (k_{X_3},f_{X_3})\geq (k_{X_4},f_{X_4})\geq\dots
$$
in $S$ which must stabilize. Therefore the original chain also stabilizes. Hence $\gl_\infty$ is $\SL_\infty$-Noetherian. The irreducible $\SL_\infty$-stable closed subsets of $\gl_\infty$ are as described in the theorem by Propositions \ref{prop_pointclosure}, \ref{prop_closedsubsetsnotcontaining}, \ref{prop_rkIleqk} and \ref{prop_closedsubsetcontaining}.
\end{proof}

\begin{re}
The techniques used in the section can also be used to generalize Theorem 1.5 from \cite{draisma-eggermont} to $G$-Noetherianity where $G=\{(g,g)\mid g\in\GL_\infty\}$. This generalization also follows from Theorem 1.2 of \cite{eggermont-snowden}.
\end{re}

\subsection{The proof of the other cases}

Now, we turn our attention to cases (2)-(4b) of Theorem \ref{thm_mainsl}. We start by proving some statements that are useful in multiple cases.

\begin{lm}\label{lm_topleftinv}
Let $k,n$ be positive integers with $k\leq n$ and let $P\in\gl_n$ be a matrix. Then $\rk(P)<k$ if and only if $\det(Q_{[k],[k]})=0$ for all $Q\sim P$.
\end{lm}
\begin{proof}
If $\rk(P)<k$, then $\det(Q_{[k],[k]})=0$ for all $Q\sim P$. Suppose that $\det(Q_{[k],[k]})=0$ for all $Q\sim P$.
Note that $\rk(P)<k$ if and only if $\det(P_{\mathscr{K},\mathscr{L}})=0$ for all subsets $\mathscr{K},\mathscr{L}\subset[n]$ of size $k$. One can prove this using reverse induction of the size of $\mathscr{K}\cap\mathscr{L}$. If $\mathscr{K}=\mathscr{L}$, then $P_{\mathscr{K},\mathscr{L}}=Q_{[k],[k]}$ for some matrix $Q\sim P$ obtained from $P$ by conjugating with a permutation matrix. So $\det(P_{\mathscr{K},\mathscr{L}})=0$. For $|\mathscr{K}\cap\mathscr{L}|<k$, we take $i\in\mathscr{K}\setminus\mathscr{L}$, $j\in\mathscr{L}\setminus\mathscr{K}$ and $\mathscr{K}'=\{j\}\cup\mathscr{K}\setminus\{i\}$ and note that, since $|\mathscr{K}'\cap\mathscr{L}|>|\mathscr{K}\cap\mathscr{L}|$,
$$
\det(P_{\mathscr{K},\mathscr{L}})=\pm\det(P_{\mathscr{K}',\mathscr{L}})\pm\det(Q_{\mathscr{K}',\mathscr{L}})=0
$$
where $Q\sim P$ is the matrix obtained from $P$ by adding row $i$ to row $j$ and substracting column $j$ from column $i$.
\end{proof}

\begin{lm}\label{lm_higherrank}
Let $k,\ell,n\in\NN$ be integers with $n\geq 6k$ and $\ell\geq2$ and let $P_1,\dots,P_\ell\in\gl_n$ be matrices of rank $k$. Then there exist $Q_1\sim P_1,\dots,Q_\ell\sim P_\ell$ such that
$$
k<\rk(Q_1+\dots+Q_\ell,I_n)=\rk(Q_1+\dots+Q_\ell)\leq 3k.
$$
\end{lm}
\begin{proof}
Let $P,P'\in\gl_n$ be matrices such that $\rk(P),\rk(P')\leq n/2-1$. We start with three claims.
\begin{itemize}
\item[(0)] For all $Q\sim P$ and $Q'\sim P'$, we have $\rk(Q+Q')\geq |\rk(P)-\rk(P')|$.
\item[(1)] There exist $Q\sim P$ and $Q'\sim P'$ with $\rk(Q+Q')=\rk(P)+\rk(P')$.
\item[(2)] There exist $Q\sim P$ and $Q'\sim P'$ with $\rk(Q+Q')\leq \max(\rk(P),\rk(P'))$.
\end{itemize}
Claim (0) is obvious. For (1) and (2), take $m=\max(\rk(P),\rk(P'))$ and note that
$$
P\sim\begin{pmatrix}TR&TRS\\R&RS\end{pmatrix}\sim\begin{pmatrix}I_m&-S\\&I_{n-m}\end{pmatrix}^{-1}\begin{pmatrix}TR&TRS\\R&RS\end{pmatrix}\begin{pmatrix}I_m&-S\\&I_{n-m}\end{pmatrix}=\begin{pmatrix}(S+T)R&0\\R&0\end{pmatrix}
$$
for some matrices $R,S,T$ with $R$ an $(n-m)\times m$ matrix of rank $\rk(P)$ by Proposition~\ref{prop_tuplerank}, because otherwise $\rk(P,I_n)<\rk(P)$ would hold. Similarly, we have 
$$
P'\sim\begin{pmatrix}\bullet&0\\R'&0\end{pmatrix}\sim\begin{pmatrix}\bullet&R''\\0&0\end{pmatrix}
$$
for some $(n-m)\times m$ matrix $R'$ and $m\times(n-m)$ matrix $R''$ that both have rank $\rk(P')$. Now (1) follows from the fact that
$$
\begin{pmatrix}\bullet&0\\R&0\end{pmatrix}+\begin{pmatrix}\bullet&R''\\0&0\end{pmatrix}
$$
has rank $\rk(P)+\rk(P')$ and (2) follows from the fact that 
$$
\begin{pmatrix}\bullet&0\\R&0\end{pmatrix}+\begin{pmatrix}\bullet&0\\R'&0\end{pmatrix}
$$
has rank at most $m$.\bigskip

Note that, since $6k\leq n$, if $Q\in\gl_n$ is a matrix with $\rk(Q)\leq 3k$, then $\rk(Q,I_n)$ equals $\rk(Q)$ as the eigenvalue $0$ must have the highest geometric multiplicity. So to prove the lemma it suffices to prove that
$$
k<\rk(Q_1+\dots+Q_\ell)\leq 3k
$$
for some $Q_1\sim P_1,\dots,Q_\ell\sim P_\ell$ using induction on $\ell$. For $\ell=2$ this follows from (1). Now suppose that $\ell>2$ and 
$$
k<\rk(Q_1+\dots+Q_{\ell-1})\leq 3k
$$
for some $Q_1\sim P_1,\dots,Q_{\ell-1}\sim P_{\ell-1}$. Using (1) if $\rk(Q_1+\dots+Q_{\ell-1})\leq 2k$ and using (0) and (2) otherwise, we see that
$$
k<\rk\left(g(Q_1+\dots+Q_{\ell-1})g^{-1}+Q_{\ell}\right))\leq 3k
$$
for some $g\in\GL_n$ and $Q_\ell\sim P_\ell$. Since $gQ_1g^{-1}\sim P_1$, \dots, $gQ_{\ell-1}g^{-1}\sim P_{\ell-1}$ and $Q_\ell\sim P_\ell$ this proves the lemma.
\end{proof}

Let $X$ be a $G$-stable closed subset of $V$ and let $X_i$ be the closure of the projection of $X$ to $\gl_{n_i}/\span(I_{n_i})$. 

\begin{lm}\label{lm_boundedtuplerankimpliesrk=0}
Suppose that $l_i+r_i\geq 2$ for all $i\in\NN$. If there exists a $k\in\ZZ_{\geq0}$ such that $X_i$ only contains elements $P\mod I_{n_i}$ with $\rk(P,I_{n_i})\leq k$ for all $i\gg0$, then $X\subseteq\{0\}$.
\end{lm}
\begin{proof}
The lemma follows by induction on $k$ from the following statement.
\begin{itemize}
\item[(*)] Let $k,i\in\NN$ be integers such that $n_i\geq 6k$. If $X_{i+1}$ contains an element $P\mod I_{n_{i+1}}$ with $\rk(P,I_{n_{i+1}})=k$, then $X_i$ contains an element $Q\mod I_{n_i}$ with $\rk(Q,I_{n_i})>k$.
\end{itemize}
Let $k,i\in\NN$ be integers such that $n_i\geq 6k$ and let $P\mod I_{n_{i+1}}$ be an element of $X_{i+1}$ with $\rk(P,I_{n_{i+1}})=k$. By replacing the representative of the element $P\mod I_{n_{i+1}}$, we may assume that $\rk(P)=k$. By Lemma \ref{lm_topleftinv}, we have
$$
gPg^{-1}=\begin{pmatrix}P_{11}&\dots&P_{1l_i}&\bullet&\dots&\bullet&\bullet\\\vdots&&\vdots&\vdots&&\vdots&\vdots\\P_{l_i1}&\dots&P_{l_il_i}&\bullet&\dots&\bullet&\bullet\\\bullet&\dots&\bullet&Q_{11}&\dots&Q_{1r_i}&\bullet\\\vdots&&\vdots&\vdots&&\vdots&\vdots\\\bullet&\dots&\bullet&Q_{r_i1}&\dots&Q_{r_ir_i}&\bullet\\\bullet&\dots&\bullet&\bullet&\dots&\bullet&\bullet\end{pmatrix}
$$
for $P_{11},\dots,P_{l_i,l_i},Q_{11},\dots,Q_{r_ir_i}\in\gl_{n_i}$ with $\rk(P_{11})=k$ for some matrix $g\in\GL_{n_{i+1}}$. Since this is an open condition on $g$, the matrix $gPg^{-1}$ is in fact of this form for sufficiently general $g\in\GL_{n_{i+1}}$. This allows us to assume that $\rk(P_{jj})=k$ for all $j\in[l_i]$ and $\rk(-Q_{\ell\ell}^T)=\rk(Q_{\ell\ell})=k$ for all $\ell\in[r_i]$. Lemma \ref{lm_higherrank} now tell us that by replacing $g$ by $\Diag(g_1,\dots,g_{l_i+r_i},I_{z_i})g$ for some $g_1,\dots,g_{l_i+r_i}\in\GL_{n_i}$, we may also assume that
$$
Q=\sum_{j=1}^{l_i}P_{jj}-\sum_{\ell=1}^{r_i}Q_{\ell\ell}^{T}
$$
satisfies $k<\rk\left(Q,I_{n_i}\right)$ and this proves (*). 
\end{proof}

Let $n\in\NN$ be a multiple of $\cha(K)$. Then the trace function on $\gl_n$ is an element of $K[\gl_n/\span(I_n)]^{\SL_n}$. Note that if $\cha(K)\mid n_i$ and $z_i=0$, then $\cha(K)\mid n_{i+1}$. So if in addition $\cha(K)=2$ or $r_i=0$, then the map
\begin{eqnarray*}
\gl_{n_{i+1}}/\span(I_{n_{i+1}})&\twoheadrightarrow&\gl_{n_i}/\span(I_{n_i})\\
\begin{pmatrix}P_{11}&\dots&P_{1l_i}&\bullet&\dots&\bullet\\\vdots&&\vdots&\vdots&&\vdots\\P_{l_i1}&\dots&P_{l_il_i}&\bullet&\dots&\bullet\\\bullet&\dots&\bullet&Q_{11}&\dots&Q_{1r_i}\\\vdots&&\vdots&\vdots&&\vdots\\\bullet&\dots&\bullet&Q_{r_i1}&\dots&Q_{r_ir_i}\end{pmatrix}\mod I_{n_{i+1}}&\mapsto&\sum_{k=1}^{l_i}P_{kk}-\sum_{\ell=1}^{r_i}Q_{\ell\ell}^{T}\mod I_{n_i}.
\end{eqnarray*}
commutes with taking the trace.

\begin{de}
When $\cha(K)\mid n_i$ and $z_i=0$ for all $i\gg0$ and in addition $\cha(K)=2$ or $r_i=0$ for all $i\gg0$, define the trace of an element $(P_i\mod I_{n_i})_i\in V$ to be the $\mu\in K$ such that $\tr(P_i)=\mu$ for all $i\gg0$. Otherwise, define the trace of any element of $V$ to be zero.
\end{de}

Note that in all cases the trace of an element of $V$ is $G$-invariant. For $\mu\in K$, denote the $G$-stable closed subset $\{P\in V\mid\tr(P)=\mu\}$ of $V$ by $Y_\mu$. Denote the closure of the projection of $Y_{\mu}$ to $\gl_{n_i}/\span(I_{n_i})$ by $Y_{\mu,i}$.

\begin{thm}\label{thm_done}
Assume that $l_i+r_i\geq2$ for all $i\in\NN$ and that $X\subsetneq Y_\mu$ for some $\mu\in K$. Suppose that for all $i\in\NN$ such that $I(Y_{\mu,i})\subsetneq I(X_i)$ and for all non-zero polynomials $f\in I(X_i)\setminus I(Y_{\mu,i})$ of minimal degree, the span of the $\SL_{n_{i+1}}$-orbit of the polynomial 
$$
f(P_{11}+\dots+P_{l_il_i}-Q_{11}^{T}-\dots-Q_{r_ir_i}^{T})\in I(X_{i+1})
$$
contains a non-zero off-diagonal polynomial. Then either $X=\emptyset$ or $X=\{0\}$.
\end{thm}
\begin{proof}
Since $X$ is strictly contained in $Y_{\mu}$, there exists an integer $j\geq 2$ such that $I(Y_{\mu,j})\subsetneq I(X_j)$. Note that $I(Y_{\mu,i})\subsetneq I(X_i)$ for all integers $i\geq j$. For all $i\geq j$, let $f_i\in I(X_i)\setminus I(Y_{\mu,i})$ be an element of minimal degree $d_i$. Then $d_i\leq d_j$ for all $i\geq j$ and by choosing $j$ large enough we may assume that $d_j\leq n_j$.\bigskip

For $i\geq j$, let $g_i\in I(X_{i+1})$ be a non-zero off-diagonal polynomial contained in the span of the $\SL_{n_{i+1}}$-orbit of $f_i(P_{11}+\dots+P_{l_il_i}-Q_{11}^{T}-\dots-Q_{r_ir_i}^{T})$. Then $\deg(g)\leq d_i\leq d_j\leq n_j\leq n_{i+1}/2$ since $n_{i+1}=(l_i+r_i)n_i+z_i\geq 2n_i$. So by Remark \ref{re_lowranklowtuplerank} and Lemma \ref{lm_boundedtuplerankimpliesrk=0}, we see that $X\subseteq\{0\}$.
\end{proof}

\begin{cor}\label{cor_done}
Assume that $l_i+r_i\geq 2$ for all $i\in\NN$. Suppose that for all $\mu\in K$, for all $G$-stable closed subsets $X\subsetneq Y_\mu$, for all $i\in\NN$ such that $I(Y_{\mu,i})\subsetneq I(X_i)$ and for all non-zero polynomials $f\in I(X_i)\setminus I(Y_{\mu,i})$ of minimal degree, the span of the $\SL_{n_{i+1}}$-orbit of the polynomial 
$$
f(P_{11}+\dots+P_{l_il_i}-Q_{11}^{T}-\dots-Q_{r_ir_i}^{T})\in I(X_{i+1})
$$
contains a non-zero off-diagonal polynomial. Then the irreducible $G$-stable closed subsets of $V$ are the non-empty subsets among $\{0\}$, $V$ and $\{v\in V\mid\tr(v)=\mu\}$ for $\mu\in K$ and every $G$-stable closed subset of $V$ is a finite union of irreducible $G$-stable closed subsets.
\end{cor}
\begin{proof}
Using Proposition \ref{prop_irr2irrinf}, it is easy to check that the mentioned subsets are either irreducible or empty. If the trace map on $V$ is zero, this is just Theorem \ref{thm_done} applied to $\mu=0$. Assume the trace map is non-zero. Then the linear map
\begin{eqnarray*}
\varphi\colon K&\rightarrow&V\\
\mu&\mapsto& ((\mu+1)E_{11}-E_{22}\mod I_{n_i})_i.
\end{eqnarray*}
has the property that $\tr(\varphi(\mu))=\mu$ for all $\mu\in K$. Let $X$ be a $G$-stable closed subset of~$V$. Then 
$$
\varphi^{-1}(X)=\left\{\mu\in K~\middle|~Y_\mu\subseteq X\right\}
$$
is a closed subset of $K$. So either $\varphi^{-1}(X)$ is finite or $\varphi^{-1}(X)=K$. By Theorem \ref{thm_done}, the intersection of $X$ with $Y_0$ is either $\emptyset$, $\{0\}$ or $Y_0$ and the intersection of $X$ with $Y_\mu$ for $\mu\in K\setminus\{0\}$ is either $\emptyset$ or $Y_\mu$. So either
$$
X=\{0\}\cup\bigcup_{\mu\in\varphi^{-1}(X)\setminus\{0\}}Y_\mu
$$
or
$$
X=\bigcup_{\mu\in\varphi^{-1}(X)}Y_\mu
$$
when $\varphi^{-1}(X)$ is finite and $X=V$ when $\varphi^{-1}(X)=K$.
\end{proof}

What remains is reduce the cases (2)-(4b) of Theorem \ref{thm_mainsl} to sequences
$$\begin{tikzcd}
\SL_{n_1} \arrow[hook,"\iota_1"]{r} & \SL_{n_2} \arrow[hook,"\iota_2"]{r} & \SL_{n_3} \arrow[hook,"\iota_3"]{r} & \dots
\end{tikzcd}$$
where the conditions of the corollary are satisfied.

\subsubsection*{Case (2): $\alpha+\beta=\gamma=\infty$}

Since $\gamma=\infty$, we do not have $z_i=0$ for all $i\gg0$. So we get $Y_0=V$ and $Y_\mu=\emptyset$ for all $\mu\in K\setminus\{0\}$. By restricting to an infinite subsequence we may assume that $l_i+r_i\geq2$ and $z_i\geq n_i$ for all $i\in\NN$. Let $i\in\NN$ be such that $I(X_i)\neq0$ and let $f\in I(X_i)$ be a non-zero polynomial of minimal degree. Take $l=l_i$, $r=r_i$, $z=z_i$, $m=n_i$ and $n=n_{i+1}=(l+r)m+z$. To prove that the conditions of Corollary~\ref{cor_done} are satisfied, we need to check the following condition:
\begin{itemize}
\item[(*)] The span of the $\SL_n$-orbit of the polynomial 
$$
g:=f(P_{11}+\dots+P_{ll}-Q_{11}^{T}-\dots-Q_{rr}^{T})
$$
contains a non-zero off-diagonal polynomial.
\end{itemize}
Consider the matrix
$$
H=\begin{pmatrix}P_{11}&\dots&P_{1l}&\bullet&\dots&\bullet&\bullet\\\vdots&&\vdots&\vdots&&\vdots&\vdots\\P_{l1}&\dots&P_{ll}&\bullet&\dots&\bullet&\bullet\\\bullet&\dots&\bullet&Q_{11}&\dots&Q_{1r}&\bullet\\\vdots&&\vdots&\vdots&&\vdots&\vdots\\\bullet&\dots&\bullet&Q_{r1}&\dots&Q_{rr}&\bullet\\R_1&\dots&R_l&\bullet&\dots&\bullet&\bullet\\\bullet&\dots&\bullet&\bullet&\dots&\bullet&\bullet\end{pmatrix}
$$
where $P_{k,\ell},Q_{k,\ell},R_k\in\gl_m$. For $\lambda\in K$, consider the matrix
$$
A(\lambda)=\begin{pmatrix}I_m&&&&&&\lambda I_m&\\&\ddots&&&&&\\&&I_m&&&\\&&&I_m&&\\&&&&\ddots&&\\&&&&&I_m\\&&&&&&I_m\\&&&&&&&I_{z-m}\end{pmatrix}.
$$
For all $\lambda\in K$, we have
$$
A(\lambda)HA(\lambda)^{-1}=\begin{pmatrix}P_{11}'&\dots&P_{1l}'&\bullet&\dots&\bullet&\bullet\\\vdots&&\vdots&\vdots&&\vdots&\vdots\\P_{l1}'&\dots&P_{ll}'&\bullet&\dots&\bullet&\bullet\\\bullet&\dots&\bullet&Q_{11}&\dots&Q_{1r}&\bullet\\\vdots&&\vdots&\vdots&&\vdots&\vdots\\\bullet&\dots&\bullet&Q_{r1}&\dots&Q_{rr}&\bullet\\\bullet&\dots&\bullet&\bullet&\dots&\bullet&\bullet\end{pmatrix}
$$
where $P_{11}'=P_{11}+\lambda R_1$ and $P_{jj}'=P_{jj}$ for all $j\in\{2,\dots,l\}$. This means that if we let $A(\lambda)$ act on $g$, we obtain the polynomial $h(\lambda)=f(P_{11}+\dots+P_{ll}-Q_{11}^T-\dots-Q_{rr}^T+\lambda R_1)$. Let $d$ be the degree of $f$ and let $f_d=f_d(P)$ be the homogeneous part of $f$ of degree~$d$. Then $f_d(R_1)$ is a non-zero off-diagonal polynomial on $\gl_n$ since $m\leq(n-1)/2$. Since $f_d(R_1)$ is the coefficient of $h(\lambda)$ at $\lambda^d$, it is contained in this span of the $h(\lambda)$. So (*) holds. So we can apply Corollary \ref{cor_done} and this proves Theorem \ref{thm_mainsl} in case (2).

\subsubsection*{Case (3a): $\beta=\infty$, $\gamma<\infty$ and $\cha(K)\neq 2$ or $2\nmid n_i$ for all $i\gg0$}

We do not have $r_i=0$ for all $i\gg0$. Furthermore, if $\cha(K)=2$, then $\cha(K)\mid n_i$ for all $i\gg0$ does not hold. So we again get $Y_0=V$ and $Y_\mu=\emptyset$ for all $\mu\in K\setminus\{0\}$. By restricting to an infinite subsequence we may assume that $r_i>0$, $l_i+r_i>2$ and $z_i=0$ for all $i\in\NN$. To assume that $l_i+r_i>2$, we use \cite[Proposition 2.4]{baranov-zhilinskii}. If $\cha(K)=2$, we may furthermore assume that $2\nmid n_i$ for all $i\in\NN$. Let $i\in\NN$ be such that $I(X_i)\neq0$ and let $f\in I(X_i)$ be a non-zero polynomial of minimal degree. Take $l=l_i$, $r=r_i$, $m=n_i$ and $n=n_{i+1}=(l+r)m$. To prove that the conditions of Corollary \ref{cor_done} are satisfied, we need to check the following condition:
\begin{itemize}
\item[(*)] The span of the $\SL_n$-orbit of the polynomial 
$$
g:=f(P_{11}+\dots+P_{ll}-Q_{11}^{T}-\dots-Q_{rr}^{T})
$$
contains a non-zero off-diagonal polynomial.
\end{itemize}
Consider the matrix
$$
H=\begin{pmatrix}P_{11}&\dots&P_{1l}&\bullet&\dots&\bullet\\\vdots&&\vdots&\vdots&&\vdots\\P_{l1}&\dots&P_{ll}&\bullet&\dots&\bullet\\R_{11}&\dots&R_{1l}&Q_{11}&\dots&Q_{1r}\\\vdots&&\vdots&\vdots&&\vdots\\R_{r1}&\dots&R_{rl}&Q_{r1}&\dots&Q_{rr}\end{pmatrix}
$$
where $P_{k,\ell},Q_{k,\ell},R_k\in\gl_m$. Also consider the matrix
$$
A(\Lambda)=\begin{pmatrix}I_m&&&\Lambda\\&\ddots\\&&I_m\\&&&I_m&&\\&&&&\ddots\\&&&&&I_m\end{pmatrix}
$$
for $\Lambda\in\gl_m$. For all $\Lambda\in\gl_m$, we have
$$
A(\Lambda)HA(\Lambda)^{-1} = \begin{pmatrix}P_{11}'&\dots&P_{1l}'&\bullet&\dots&\bullet\\\vdots&&\vdots&\vdots&&\vdots\\P_{l1}'&\dots&P_{ll}'&\bullet&\dots&\bullet\\\bullet&\dots&\bullet&Q_{11}'&\dots&Q_{1r}'\\\vdots&&\vdots&\vdots&&\vdots\\\bullet&\dots&\bullet&Q_{r1}'&\dots&Q_{rr}'\end{pmatrix}
$$
where
\begin{eqnarray*}
P_{11}'&=&P_{11}+\Lambda R_{11}\\
P_{jj}'&=&P_{jj}\mbox{ for $j\in\{2,\dots,l\}$}\\
Q_{11}'&=&Q_{11}-R_{11}\Lambda\\
Q_{jj}'&=&Q_{\ell\ell}\mbox{ for $\ell\in\{2,\dots,r\}$}. 
\end{eqnarray*}
This means that if we let $A(\Lambda)$ act on the polynomial $g$, we obtain the polynomial $h(\Lambda)=f(P_{11}+\dots+P_{ll}-Q_{11}^T-\dots-Q_{rr}^T+\Lambda R_{11}+\Lambda^T R_{11}^T)$. Let $d$ be the degree of $f$ and let $f_d=f_d(P)$ be the homogeneous part of $f$ of degree $d$. Then we see that the homogeneous part of $h(\Lambda)$ of degree $d$ in the coordinates of $\Lambda$ equals $f_d(\Lambda R_{11}+\Lambda^T R_{11}^T)$.\bigskip

To prove that $f_d(\Lambda R_{11}+\Lambda^T R_{11}^T)$ is non-zero as a polynomial in $\Lambda$ and $R_{11}$, we will use reduction rules for graphs. See for example \cite{bodlaender-antwerpen} for more on this. Let $\Gamma$ be an undirected multigraph. Denote its vertex and edge sets by $V(\Gamma)$ and $E(\Gamma)$.

\medskip
\begin{de}~
We consider the following three reduction rules:
\begin{itemize}
\item[(1)] Remove an edge from $\Gamma$.
\item[(2)] Remove a vertex of $\Gamma$ that has at least one loop.
\item[(3)] Pick a vertex $v$ of $\Gamma$ that has a least one loop. Replace an edge of $\Gamma$ with endpoints $v\neq w$ by a loop at $w$. 
\end{itemize}
We say that $\Gamma$ reduces to a multigraph $\Gamma'$ if $\Gamma'$ can be obtained from $\Gamma$ by applying a series of reductions.
\end{de}

\begin{lm}
If $\Gamma$ reduces to the empty graph, then the linear map
\begin{eqnarray*}
\ell_\Gamma\colon K^{E(\Gamma)}&\rightarrow&K^{V(\Gamma)}\\
(x_e)_e&\mapsto&\left(\sum_{e\ni v}x_e\right)_v
\end{eqnarray*}  
is surjective. Here entries corresponding to loops are only added once.
\end{lm}
\begin{proof}
If $\Gamma$ is the empty graph, then $\ell_\Gamma$ is surjective. So it suffices to check that $\ell_\Gamma$ is surjective whenever we have a reduction $\Gamma'$ of $\Gamma$ such that the similarly defined map $\ell_{\Gamma'}$ is surjective. When $\Gamma'$ is obtained from $\Gamma$ by applying reduction rule (1), this is easy. The other cases follow from the fact that $x_e$ only appears in coordinate $v$ when $e$ is a loop with endpoint $v$.
\end{proof}

\smallskip

\begin{lm}~\label{lm_char2}
\begin{itemize}
\item[(a)] If $\cha(K)\neq 2$, then $\{PQ+P^TQ^T\mid P,Q\in\gl_n\}=\gl_n$ for all $n\in\NN$.
\item[(b)] If $\cha(K)=2$, then $\{PQ+P^TQ^T\mid P,Q\in\gl_n\}$ is dense in $\sl_n$ for all $n\in\NN$.
\end{itemize}
\end{lm} 
\begin{proof}
In part (a) we can even take $P$ and $Q$ to be symmetric, because by \cite[(ii)]{taussky} every matrix is a product of two symmetric matrices. For part (b), suppose that $\cha(K)=2$ and let $n\in\NN$ be an integer. Then $PQ+P^TQ^T\in\sl_n$ for all $P,Q\in\gl_n$. Note that $\{PQ+P^TQ^T\mid P,Q\in\gl_n\}$ is dense in $\sl_n$ if and only if the morphism
\begin{eqnarray*}
\varphi\colon \gl_n\times\gl_n&\rightarrow&\gl_n/\span(E_{n,n})\\
(P,Q)&\mapsto& PQ+P^TQ^T\mod E_{n,n}
\end{eqnarray*}
is dominant. To show that $\varphi$ is dominant, it suffices to show that its derivative 
\begin{eqnarray*}
\mathrm{d}_{(R,S)}\varphi\colon \gl_n\oplus\gl_n&\rightarrow&\gl_n/\span(E_{n,n})\\
(P,Q)&\mapsto& PS + P^TS^T + RQ + R^TQ^T \mod E_{n,n}
\end{eqnarray*}
at the point
$$
(R,S)=\left(\begin{pmatrix}0&1&&\\&\ddots&\ddots&\\&&\ddots&1\\&&&0\end{pmatrix},\begin{pmatrix}&&&&1\\&&&\iddots\\&&\iddots&\\&\iddots&&\\1&&&\end{pmatrix}\right)
$$
is surjective. Note that
\begin{eqnarray*}
(\mathrm{d}_{(R,S)}\varphi)(E_{i,j},0)&=&E_{i,n+1-j}+E_{j,n+1-i}\\
(\mathrm{d}_{(R,S)}\varphi)(0,E_{k,\ell})&=&(1-\delta_{k1})E_{k-1,\ell}+(1-\delta_{\ell n})E_{\ell+1,k}
\end{eqnarray*}
and hence $(\mathrm{d}_{(R,S)}\varphi)(0,E_{1,n})=0$ and $(\mathrm{d}_{(R,S)}\varphi)(E_{i,i},0)=0$ for all $i\in[n]$, because $\cha(K)=2$. The other basis elements of $\gl_n\oplus\gl_n$ all get sent to a sum of one or two basis elements of $\gl_n/\span(E_{n,n})$. To prove that $\mathrm{d}_{(R,S)}\varphi$ is surjective, it suffices by the previous lemma to prove that the restriction of $\mathrm{d}_{(R,S)}\varphi$ to the span of these other basis vectors equals $\ell_\Gamma$ for some multigraph $\Gamma$ that reduces to the empty graph.\bigskip

Define the multigraph $\Gamma$ as follows: We let $V(\Gamma)$ be the basis $\{E_{i,j}\mid (i,j)\neq(n,n)\}$ of $\gl_n/\span(E_{n,n})$ and we let $E(\Gamma)$ be the set
$$
\{(E_{i,j},0)\mid i\neq j\}\cup\{(0,E_{k,\ell})\mid k,\ell\in[n]\}\setminus\{(0,E_{1,n})\}
$$
of basis element of $\gl_n\oplus\gl_n$ that are not mapped to $0$. This allows to define the set of endpoints of an edge in such a way that $(\mathrm{d}_{(R,S)}\varphi)|_{\span(E(\Gamma))}=\ell_{\Gamma}$. Next we check that $\Gamma$ reduces to the empty graph. One can check that $\Gamma$ has two loops at $E_{1,1}$, a loop at $E_{k,1}$ for all $k>1$ and a loop at $E_{\ell,n}$ for all $\ell<n$. We also have:
\begin{itemize}
\item[(x)] edges with endpoints $E_{i,j}$ and $E_{j+1,i+1}$ for all $i,j\in[n-1]$;
\item[(y)] edges with endpoints $E_{k,1}$ and $E_{n,n+1-k}$ for all $1<k<n$; and
\item[(z)] edges with endpoints $E_{\ell,n}$ and $E_{1,n+1-\ell}$ for $1<\ell<n$.
\end{itemize}
First, we remove all other edges from $\Gamma$ using reduction rule (1). Next, we replace the edges (y) and (z) by loops at $E_{n,k}$ for $1<k<n$ and $E_{1,\ell}$ for $1<\ell<n$ using reduction rule (3). The graph $\Gamma'$ obtained this way has has the edges (x) together with loops at $E_{1,1}$ and $E_{1,i},E_{n,i},E_{i,1},E_{i,n}$ for $1<i<n$. Now consider the connected components of $\Gamma'$. One connected component consists of a path from $E_{1,1}$ to $E_{n,n}$ with a loop at $E_{1,1}$. All other components are path with loops at both ends starting at a vertex of the form $E_{1,i}$ or $E_{i,1}$ and ending at a vertex of the form $E_{n,i}$ or $E_{i,n}$. Each of these components reduces to the empty graph by repeatedly using reduction rules (2) and (3). Therefore $\Gamma'$ and $\Gamma$ also reduce to the empty graph. Hence $\mathrm{d}_{(R,S)}\varphi$ is surjective and $\varphi$ is dominant.
\end{proof} 

Since the polynomial $f$ is non-zero, so is $f_d$. By combining the lemma with the fact that $f_d(P+\lambda I_m)=f_d(P)$ for all $P\in\gl_m$ and $\lambda\in K$, we see that the polynomial $f_d(\Lambda R_{11}+\Lambda^T R_{11}^T)$ is non-zero. Now view $f_d(\Lambda R_{11}+\Lambda^T R_{11}^T)$ as a polynomial in~$\Lambda$ whose coefficients are polynomials in the entries of $R_{11}$. Any of its non-zero coefficients is a non-zero off-diagonal polynomial on $\gl_n$ which is contained in the span of the orbit of $g$. Here we use that $m\leq(n-1)/2$ since $l+r>2$. So (*) holds. So we can apply Corollary \ref{cor_done} and this proves Theorem \ref{thm_mainsl} in case (3a).

\subsubsection*{Case (3b): $\beta=\infty$, $\gamma<\infty$, $\cha(K)=2$ and $2\mid n_i$ for all $i\gg0$}

Note that in this case the trace map on $V$ is non-zero. By restricting to an infinite subsequence we may assume that $r_i>0$, $l_i+r_i>2$, $z_i=0$ and $2\mid n_i$ for all $i\in\NN$. Let $\mu\in K$, suppose that $X\subsetneq Y_\mu$ and let $i\in\NN$ be such that $I(Y_{\mu,i})\subsetneq I(X_i)$. Let $f\in I(X_i)\setminus I(Y_{\mu,i})$ be a polynomial of minimal degree. Take $l=l_i$, $r=r_i$, $m=n_i$ and $n=n_{i+1}=(l+r)n$. To prove that the conditions of Corollary \ref{cor_done} are satisfied, we need to check the following condition:
\begin{itemize}
\item[(*)] The span of the $\SL_n$-orbit of the polynomial 
$$
g:=f(P_{11}+\dots+P_{ll}-Q_{11}^{T}-\dots-Q_{rr}^{T})
$$
contains a non-zero off-diagonal polynomial.
\end{itemize}
As in case (3a), we find that all coefficients of $f_d(\Lambda R_{11}+\Lambda^T R_{11}^T)$ are off-diagonal polynomials on $\gl_n$ which are contained in the span of the orbit of $g$. So it suffices to prove that $f_d(\Lambda R_{11}+\Lambda^T R_{11}^T)$ is not the zero polynomial.\bigskip

Suppose that the polynomial $f_d(\Lambda R_{11}+\Lambda^T R_{11}^T)$ is the zero polynomial. Then $f_d(P)=0$ for all $P\in\sl_m$ by Lemma \ref{lm_char2}(b). So $f_d$ is a multiple of the trace function on~$\gl_m$ and we can write $f_d=\tr\cdot h$ for some $h$. But then $f-(\tr-\mu)h\in I(X_i)\setminus I(Y_{\mu,i})$. This contradicts the minimality of the degree of $f$. So $f_d(\Lambda R_{11}+\Lambda^T R_{11}^T)$ can not be the zero polynomial. So (*) again holds. So we can apply Corollary \ref{cor_done} and this proves Theorem \ref{thm_mainsl} in case (3b).

\subsubsection*{Case (4a): $\beta+\gamma<\infty$ and $\cha(K)\nmid n_i$ for all $i\gg0$}

We do not have $\cha(K)\nmid n_i$ for all $i\gg0$. So we get $Y_0=V$ and $Y_\mu=\emptyset$ for all $\mu\in K\setminus\{0\}$. By restricting to an infinite subsequence we may assume that $l_i>2$, $r_i=z_i=0$ and $\cha(K)\nmid n_i$ for all $i\in\NN$. Let $i\in\NN$ be such that $I(X_i)\neq0$ and let $f\in I(X_i)$ be a non-zero polynomial of minimal degree. Take $l=l_i$, $m=n_i$ and $n=n_{i+1}=lm$. Then $m\leq (n-1)/2$. To prove that the conditions of Corollary \ref{cor_done} are satisfied, we need to check the following condition:
\begin{itemize}
\item[(*)] The span of the $\SL_n$-orbit of the polynomial 
$$
g:=f(P_{11}+\dots+P_{ll})
$$
contains a non-zero off-diagonal polynomial.
\end{itemize}
Consider the matrix
$$
H=\begin{pmatrix}P_{11}&\dots&P_{1l}\\\vdots&&\vdots\\P_{l1}&\dots&P_{ll}\end{pmatrix}
$$
where $P_{k,\ell}\in\gl_m$. Also consider the matrix
$$
A(\Lambda)= \begin{pmatrix}I_m&\Lambda\\&I_m\\&&I_m&&\\&&&\ddots\\&&&&I_m\end{pmatrix}
$$
for $\Lambda\in\gl_m$. For all $\Lambda\in\gl_m$, we have
$$
A(\Lambda)HA(\Lambda)^{-1} = \begin{pmatrix}P_{11}'&\dots&P_{1l}'\\\vdots&&\vdots\\P_{l1}'&\dots&P_{ll}'\end{pmatrix}
$$
where $P_{11}'=P_{11}+\Lambda P_{21}$, $P_{22}'=P_{22}-P_{21}\Lambda$ and $P_{jj}'=P_{jj}$ for $j\in\{3,\dots,l\}$. This means that if we let $A(\Lambda)$ act on $g$, we obtain the polynomial $h(\Lambda)=f(P_{11}+\dots+P_{ll}+[\Lambda,P_{21}])$ where $[-,-]$ is the commutator bracket. Let $d$ be the degree of $f$ and let $f_d=f_d(P)$ be the homogeneous part of $f$ of degree $d$. Then we see that the homogeneous part of $h(\Lambda)$ of degree $d$ in the coordinates of $\Lambda$ equals $f_d([\Lambda,P_{21}])$. Since $f$ is non-zero, so is $f_d$. By \cite[Theorem 6.3]{stasinski}, we know that every element of $\gl_m$ is of the form $[X,Y]+\lambda I_m$ for some $X,Y\in\gl_m$ and $\lambda\in K$. So since $f_d(P+\lambda I_m)=f_d(P)$ for all $P\in\gl_m$ and $\lambda\in K$, we see that $f_d([\Lambda,P_{21}])$ is not the zero polynomial. Any non-zero coefficient of $f_d([\Lambda,P_{21}])$ as a polynomial in $\Lambda$ satisfies (*). So we can apply Corollary \ref{cor_done} and this proves Theorem \ref{thm_mainsl} in case (4a).

\subsubsection*{Case (4b): $\beta+\gamma<\infty$ and $\cha(K)\mid n_i$ for all $i\gg0$}

Note that in this case the trace map on $V$ is non-zero. By restricting to an infinite subsequence we may assume that $l_i>2$, $r_i=z_i=0$ and $\cha(K)\mid n_i$ for all $i\in\NN$. We now proceed as in the case (4a) with the same modifications that were established in case (3b). 

\section{Limits of classical groups of type C}

From now on, we assume that $\cha(K)\neq2$. In this section, we let $G$ be the direct limit of a sequence 
$$\begin{tikzcd}
\Sp_{2n_1} \arrow[hook,"\iota_1"]{r} & \Sp_{2n_2} \arrow[hook,"\iota_2"]{r} & \Sp_{2n_3} \arrow[hook,"\iota_3"]{r} & \dots
\end{tikzcd}$$
of diagonal embeddings given by
\begin{eqnarray*}
\iota_i\colon \Sp_{2n_i}&\hookrightarrow&\Sp_{2n_{i+1}}\\
\begin{pmatrix}A&B\\C&D\end{pmatrix}&\mapsto&\begin{pmatrix}\Diag(A,\dots,A,I_{z_i})&\Diag(B,\dots,B,0)\\\Diag(C,\dots,C,0)&\Diag(D,\dots,D,I_{z_i})\end{pmatrix}
\end{eqnarray*}
with $l_i$ blocks $A,B,C,D\in\gl_{n_i}$ for some $l_i\in\NN$ and $z_i\in\ZZ_{\geq0}$. We let $V$ be the inverse limit of the sequence
$$\begin{tikzcd}
\sp_{2n_1}& \arrow[two heads]{l} \sp_{2n_2} & \arrow[two heads]{l} \sp_{2n_3} & \arrow[two heads]{l} \dots
\end{tikzcd}$$
where the maps are given by
\begin{eqnarray*}
\sp_{2n_{i+1}}&\twoheadrightarrow&\sp_{2n_i}\\
\begin{pmatrix}P_{11}&\dots&P_{1l_i}&\bullet&Q_{11}&\dots&Q_{1l_i}&\bullet\\\vdots&&\vdots&\vdots&\vdots&&\vdots&\vdots\\P_{l_i1}&\dots&P_{l_il_i}&\vdots&Q_{l_i1}&\dots&Q_{l_il_i}&\vdots\\\bullet&\dots&\dots&\bullet&\bullet&\dots&\dots&\bullet\\R_{11}&\dots&R_{1l_i}&\bullet&S_{11}&\dots&S_{1l_i}&\bullet\\\vdots&&\vdots&\vdots&\vdots&&\vdots&\vdots\\R_{l_i1}&\dots&R_{l_il_i}&\vdots&S_{l_i1}&\dots&S_{l_il_i}&\vdots\\\bullet&\dots&\dots&\bullet&\bullet&\dots&\dots&\bullet\end{pmatrix}&\mapsto&\begin{pmatrix}\sum_{k=1}^{l_i}P_{kk}&\sum_{k=1}^{l_i}Q_{kk}\\\sum_{k=1}^{l_i}R_{kk}&\sum_{k=1}^{l_i}S_{kk}\end{pmatrix}
\end{eqnarray*}
with $P_{k\ell}=-S_{\ell k}^T,Q_{k\ell},R_{k\ell}\in\gl_{n_i}$ such that $Q_{k\ell}=Q_{\ell k}^T$ and $R_{k\ell}=R_{\ell k}^T$.

\begin{thm}\label{thm_mainsp}
The space $V$ is $G$-Noetherian.
\end{thm}

Let $X\subsetneq V$ be a $G$-stable closed subset. Let $X_i$ be the closure of the projection of $X$ to $\sp_{2n_i}$ and let $I(X_i)\subseteq K[\sp_{2n_i}]$ be the ideal of $X_i$. If $\#\{i\mid l_i>1\}<\infty$, then Theorem \ref{thm_mainsp} follows from \cite[Theorem 1.2]{eggermont-snowden}.

\begin{re}
Let $X\subsetneq V$ be a $G$-stable closed subset in the case where $\#\{i\mid l_i>1\}<\infty$. Then $V$ can be identified with a subspace of the space of $\NN\times\NN$ matrices and we can prove (using technique similar to the ones used in this paper) that $X$ consists of matrices of bounded rank. The $G$-Noetherianity of $V$ then follows from the $\Sym(\NN)$-Noetherianity of $K^{\NN\times k}$ for $k\in\NN$. Important to note here is that, for every $n\in\NN$, the group $\Sp_{2n}$ contains all matrices corresponding to permutations $\pi\in S_{2n}$ such that $\pi(i+n)=\pi(i)+n$ for all $i\in[n]$. This allows us to define an action of $\Sym(\NN)$ on $V$, up to which the closed subset $X$ is Noetherian. Similar statements hold for sequences of types $B$ and $D$.
\end{re}

We assume that $\#\{i\mid l_i>1\}=\infty$. By restricting to an infinite subsequence, we may assume that $l_i\geq 3$ for all $i\in\NN$.

\begin{lm}\label{lm_toppartnotinsp}
Let $n\in\NN$, let $Y\subsetneq\sp_{2n}$ be an $\Sp_{2n}$-stable closed subset and let $Z$ be the closed subset
$$
\left\{\begin{pmatrix}P&Q\\R&-P^T\end{pmatrix}\in\sp_{2n}~\middle|~P=P^T\right\}
$$
of $\sp_{2n}$. Then there is a non-zero polynomial $f\in I(Y)$ whose top-graded part is not contained in the ideal of $Z$.
\end{lm}
\begin{proof}
Since $Y\subsetneq\sp_{2n}$, there is a non-zero polynomial $f\in I(Y)$. Since $f$ is non-zero, so is its top-graded part $g$. Let the group $\GL_n$ act on $\sp_{2n}$ via the diagonal embedding $\GL_n\hookrightarrow\Sp_{2n},A\mapsto\Diag(A,A^{-T})$. Then we get a action of $\GL_n$ on $K[\sp_{2n}]$. Note that this action respects the grading on $K[\sp_{2n}]$ and that the ideal $I(Y)$ is $\GL_n$-stable. So for all $A\in\GL_n$ we have $A\cdot f\in I(Y)$ and the top-graded part of this polynomial is $A\cdot g$. Hence it suffices to prove that $A\cdot g\not\in I(Z)$ for some $A\in\GL_n$. Note that
\begin{eqnarray*}
\GL_n\cdot Z&=&\left\{A\cdot \begin{pmatrix}P&Q\\R&-P^T\end{pmatrix}~\middle|~\begin{array}{c}P=P^T,A\in\GL_n\\Q=Q^T,R=R^T\end{array} \right\}\\
&=&\left\{\begin{pmatrix}APA^{-1}&AQA^T\\A^{-T}RA^{-1}&-A^{-T}P^TA^T\end{pmatrix}~\middle|~\begin{array}{c}P=P^T,A\in\GL_n\\Q=Q^T,R=R^T\end{array} \right\}\\
&=&\left\{\begin{pmatrix}APA^{-1}&Q\\R&-(APA^{-1})^T\end{pmatrix}~\middle|~\begin{array}{c}P=P^T,A\in\GL_n\\Q=Q^T,R=R^T\end{array} \right\}
\end{eqnarray*}
and that $\{APA^{-1}\mid P=P^T,A\in\GL_n\}$ is dense in $\gl_n$ since $K$ is infinite and diagonal matrices are symmetric. So $\GL_n\cdot Z$ is dense in $\sp_{2n}$. So since the polynomial $g$ is non-zero, there must be an $A\in\GL_n$ such that $A\cdot g\not\in I(Z)$.
\end{proof}

\begin{lm}\label{lm_movingeqsp2}
Let $i\in\NN$ and let $f=f(P,Q,R)\in I(X_i)$ be a non-zero polynomial whose top-graded part $g$ is not contained in the ideal of
$$
\left\{\begin{pmatrix}P&Q\\R&-P^T\end{pmatrix}\in\sp_{2n_i}~\middle|~P=P^T\right\}.
$$
Then $I(X_{i+1})\cap K[r_{k\ell}|1\leq k,\ell\leq n_{i+1}]/(r_{k\ell}-r_{\ell k})$ contains a non-zero polynomial with degree at most $\deg(f)$.
\end{lm}
\begin{proof}
Take $m=n_i$, $l=l_i$, $z=z_i$ and $n=n_{i+1}=lm+z$. Consider the matrix
$$
H=\begin{pmatrix}P_{11}&\dots&P_{1l}&\bullet&Q_{11}&\dots&Q_{1l}&\bullet\\\vdots&&\vdots&\vdots&\vdots&&\vdots&\vdots\\P_{l1}&\dots&P_{ll}&\vdots&Q_{l1}&\dots&Q_{ll}&\vdots\\\bullet&\dots&\dots&\bullet&\bullet&\dots&\dots&\bullet\\R_{11}&\dots&R_{1l}&\bullet&S_{11}&\dots&S_{1l}&\bullet\\\vdots&&\vdots&\vdots&\vdots&&\vdots&\vdots\\R_{l1}&\dots&R_{ll}&\vdots&S_{l1}&\dots&S_{ll}&\vdots\\\bullet&\dots&\dots&\bullet&\bullet&\dots&\dots&\bullet\end{pmatrix}\in\sp_{2n}
$$
and consider the matrix
$$
A(\lambda)=\begin{pmatrix}I_m&&&&\lambda I_m\\&I_m&&\lambda I_m\\&&I_{n-2m}\\&&&I_m\\&&&&I_m\\&&&&&I_{n-2m}\end{pmatrix}\in\Sp_{2n}
$$
for $\lambda\in K$. The polynomial $f=f(P,Q,R)\in I(X_i)$ pulls back to the element
$$
f\left(\sum_{k=1}^lP_{kk},\sum_{k=1}^lQ_{kk},\sum_{k=1}^lR_{kk}\right)
$$
of $I(X_{i+1})$. For $\lambda\in K$, we have
$$
A(\lambda)HA(\lambda)^{-1}=\begin{pmatrix}P_{11}'&\dots&P_{1l}'&\bullet&Q_{11}'&\dots&Q_{1l}'&\bullet\\\vdots&&\vdots&\vdots&\vdots&&\vdots&\vdots\\P_{l1}'&\dots&P_{ll}'&\vdots&Q_{l1}'&\dots&Q_{ll}'&\vdots\\\bullet&\dots&\dots&\bullet&\bullet&\dots&\dots&\bullet\\R_{11}&\dots&R_{1l}&\bullet&S_{11}'&\dots&S_{1l}'&\bullet\\\vdots&&\vdots&\vdots&\vdots&&\vdots&\vdots\\R_{l1}&\dots&R_{ll}&\vdots&S_{l1}'&\dots&S_{ll}'&\vdots\\\bullet&\dots&\dots&\bullet&\bullet&\dots&\dots&\bullet\end{pmatrix}
$$
where
\begin{eqnarray*}
P_{11}'&=& P_{11}+\lambda R_{21}\\
P_{22}'&=& P_{22}+\lambda R_{12}\\
P_{kk}'&=& P_{kk}\mbox{ for $k=3,\dots,l$}\\
Q_{11}'&=& Q_{11}+\lambda(S_{21}-P_{12})-\lambda^2R_{22}\\
Q_{22}'&=& Q_{22}+\lambda(S_{12}-P_{21})-\lambda^2R_{11}\\
Q_{kk}'&=& Q_{kk}\mbox{ for $k=3,\dots,l$}
\end{eqnarray*}
Let $g$ be the top-graded part of $f$. Then we see that $g(R_{21}+R_{12},-(R_{11}+R_{22}),\sum_{k=1}^lR_{kk})$ is contained in the span of 
$$
A(\lambda)\cdot f\left(\sum_{k=1}^lP_{kk},\sum_{k=1}^lQ_{kk},\sum_{k=1}^lR_{kk}\right)
$$
over all $\lambda\in K$. We have $g(P,Q,R)\neq0$ for some symmetric matrices $P,Q,R\in\gl_m$. Since $\cha(K)\neq 2$, there are matrices $R_{12},R_{21}$ such that $R_{12}=R_{21}^T$ and $R_{21}+R_{12}=P$. And, since $l>2$, there are symmetric matrices $R_{11},\dots,R_{ll}$ such that $-(R_{11}+R_{22})=Q$ and $\sum_{k=1}^lR_{kk}=R$. So we see that the polynomial 
$$
g\left(R_{21}+R_{12},-(R_{11}+R_{22}),\sum_{k=1}^lR_{kk}\right)\in I(X_{i+1})
$$
is non-zero.
\end{proof}

Since $X\subsetneq V$, we know that $X_j\subsetneq\sp_{2n_j}$ for some $j\in\NN$. Using the previous lemma, we see that there is a $d\in\ZZ_{\geq0}$ such that $I(X_i)\cap K[r_{k\ell}|1\leq k,\ell\leq n_i]/(r_{k\ell}-r_{\ell k})$ contains a non-zero polynomial of degree at most $d$ for all $i>j$.

\begin{lm}\label{lm_boundedranksp}
Let $n\in\NN$, let $Y\subsetneq\sp_{2n}$ be an $\Sp_{2n}$-stable closed subset, let
$$
M=\begin{pmatrix}M_{11}&M_{12}\\M_{21}&M_{22}\end{pmatrix}\in Y
$$
be an element and suppose that 
$$
I(Y)\cap K[r_{k\ell}|1\leq k,\ell\leq n]/(r_{k\ell}-r_{\ell k})
$$
contains a non-zero polynomial of degree $m+1$. Then $\rk(M_{12}),\rk(M_{21})\leq m$. Furthermore, if $n>6m$, then $\rk(M_{11})=\rk(M_{22})\leq 3m/2$ and $\rk(M)\leq 5m$.
\end{lm}
\begin{proof}
Let $\GL_n$ act on $\sp_{2n}$ via the diagonal embedding
\begin{eqnarray*}
\GL_n&\hookrightarrow&\Sp_{2n}\\
g&\mapsto&\Diag(g,g^{-T})
\end{eqnarray*}
and on $\{R\in\gl_n\mid R=R^T\}$ by $g\cdot R=g^{-T}Rg^{-1}$. Then the projection map
\begin{eqnarray*}
\pi\colon\sp_{2n}&\to&\gl_n\\
\begin{pmatrix}P&Q\\R&S\end{pmatrix}&\mapsto&R
\end{eqnarray*}
is $\GL_n$-equivairant. Let $Z$ be the closure of $\pi(Y)$ in $\{R\in\gl_n\mid R=R^T\}$. Since $Y$ is $\GL_n$-stable, so are $\pi(Y)$ and $Z$. Since $\cha(K)\neq2$, the $\GL_n$-orbits of $\{R\in\gl_n\mid R=R^T\}$ consist of all symmetric matrices of equal rank. So $Z$ must consist of all symmetric matrices of rank at most $h$ for some $h\leq n$. Since $I(Z)$ contains a non-zero polynomial of degree $m+1$, we see that $h\leq m$. See, for example, \cite[\S4]{sturmfels-sullivant}. So 
$$
Y\subseteq\left\{\begin{pmatrix}P&Q\\R&S\end{pmatrix}\in\sp_{2n}~\middle|~\rk(R)\leq m\right\}.
$$
Let $A\in\gl_n$ be a symmetric matrix. Then we have
$$
\begin{pmatrix}0&I_n\\-I_n&A\end{pmatrix}\in\Sp_{2n}
$$
with inverse 
$$
\begin{pmatrix}A&-I_n\\I_n&0\end{pmatrix}.
$$
Let
$$
\begin{pmatrix}P&Q\\R&S\end{pmatrix}
$$
be an element of $Y$. Then
$$
\begin{pmatrix}0&I_n\\-I_n&A\end{pmatrix}\begin{pmatrix}P&Q\\R&S\end{pmatrix} \begin{pmatrix}0&I_n\\-I_n&A\end{pmatrix}^{-1}=\begin{pmatrix}\bullet&\bullet\\ARA+AS-PA-Q&\bullet\end{pmatrix}\in Y.
$$
So we get $\rk(ARA+AS-PA-Q)\leq m$ for all symmetric matrices $A\in\gl_n$. For $A=0$, this gives us $\rk(Q)\leq m$ and so $\rk(M_{12})\leq m$ in particular. For all $A$, we can write
$$
PA+(PA)^T=(ARA+AS-PA-Q)-ARA+Q
$$
since $S=-P^T$. We get
$$
\rk(PA+(PA)^T)\leq \rk(ARA+AS-PA-Q)+\rk(ARA)+\rk(Q)\leq 3m.
$$
Since we had no conditions on the element
$$
\begin{pmatrix}P&Q\\R&S\end{pmatrix}\in Y,
$$
we also get $\rk(P'A+(P'A)^T)\leq 3m$ for all 
$$
\begin{pmatrix}P'&\bullet\\\bullet&\bullet\end{pmatrix}\in \GL_n\cdot\begin{pmatrix}P&Q\\R&S\end{pmatrix}\subseteq Y
$$
and hence $\rk(P'A+(P'A)^T)\leq 3m$ for all $P'\sim P$. Now assume that $n>6m$. Choose $A=\Diag(I_{2m+1},0)$ and write
$$
P'=\begin{pmatrix}P_{11}'&P_{12}'\\P_{21}'&P_{22}'\end{pmatrix}\sim P
$$
with $P_{21}'\in\gl_{2m+1}$. Then 
$$
P'A+(P'A)^T=\begin{pmatrix}\bullet&\bullet&P_{21}'^T\\\bullet\\P_{21}'\end{pmatrix}
$$
and hence $\rk(P_{21}')\leq 3m/2$. By Proposition \ref{prop_tuplerank}, we see that $\rk(P,I_n)\leq 3m/2$ and hence $\rk(P+\lambda I_n)\leq 3m/2$ for some $\lambda\in K$. Next, choose $A=I_n$. Then we see that $\rk(P+P^T)\leq 3m$. So
$$
\rk(2\lambda I_n)\leq \rk(P+P^T)+\rk(P+\lambda I_n)+\rk(P^T+\lambda I_n)\leq 6m<n
$$
and hence $\lambda=0$. So we in fact have $\rk(P)\leq 3m/2$. In particular, we see that $\rk(M_{11})=\rk(M_{22})\leq 3m/2$. Combining this with $\rk(M_{12}),\rk(M_{21})\leq m$, we get $\rk(M)\leq 5m$.
\end{proof}

Using Lemma \ref{lm_boundedranksp}, we see that there is an $m\in\ZZ_{\geq0}$ such that 
$$
X_i\subseteq\left\{\begin{pmatrix}P&Q\\R&S\end{pmatrix}\in\sp_{2n}~\middle|~\rk(P)\leq m\right\}
$$
for all $i\gg0$. As in the proof of Lemma \ref{lm_boundedtuplerankimpliesrk=0}, we see using Lemma \ref{lm_higherrank} that this in fact holds for $m=0$.

\begin{lm}
Let $n\in\NN$ and let $Y\subsetneq\sp_{2n}$ be an $\Sp_{2n}$-stable closed subset of 
$$
\left\{\begin{pmatrix}0&Q\\R&0\end{pmatrix}~\middle|~\begin{array}{c}Q\in\gl_n,Q=Q^T\\R\in\gl_n,R=R^T\end{array}\right\}.
$$
Then $Y\subseteq\{0\}$.
\end{lm}
\begin{proof}
Let 
$$
\begin{pmatrix}0&Q\\R&0\end{pmatrix}
$$
be an element of $Y$. Then 
$$
\begin{pmatrix}0&I_n\\-I_n&I_n\end{pmatrix}\begin{pmatrix}0&Q\\R&0\end{pmatrix}\begin{pmatrix}0&I_n\\-I_n&I_n\end{pmatrix}^{-1}=\begin{pmatrix}R&\bullet\\\bullet&\bullet\end{pmatrix}\in Y
$$
since $Y$ is $\Sp_{2n}$-stable and therefore $R=0$. By Lemma \ref{lm_boundedranksp}, we see that $Q=0$. 
\end{proof}

The lemma shows that $X\subseteq\{0\}$. So when $\#\{i\mid l_i>1\}=\infty$, the only $G$-stable closed subsets of $V$ are $V$, $\{0\}$ and $\emptyset$. This proves in particular that $V$ is $G$-Noetherian.

\section{Limits of classical groups of type D}

Recall that we assume that $\cha(K)\neq2$. In this section, we let $G$ be the direct limit of a sequence 
$$\begin{tikzcd}
\O_{2n_1} \arrow[hook,"\iota_1"]{r} & \O_{2n_2} \arrow[hook,"\iota_2"]{r} & \O_{2n_3} \arrow[hook,"\iota_3"]{r} & \dots
\end{tikzcd}$$
of diagonal embeddings given by
\begin{eqnarray*}
\iota_i\colon \O_{2n_i}&\hookrightarrow&\O_{2n_{i+1}}\\
\begin{pmatrix}A&B\\C&D\end{pmatrix}&\mapsto&\begin{pmatrix}\Diag(A,\dots,A,I_{z_i})&\Diag(B,\dots,B,0)\\\Diag(C,\dots,C,0)&\Diag(D,\dots,D,I_{z_i})\end{pmatrix}
\end{eqnarray*}
with $l_i$ blocks $A,B,C,D\in\gl_{n_i}$ for some $l_i\in\NN$ and $z_i\in\ZZ_{\geq0}$. We let $V$ be the inverse limit of the sequence
$$\begin{tikzcd}
\o_{2n_1}& \arrow[two heads]{l} \o_{2n_2} & \arrow[two heads]{l} \o_{2n_3} & \arrow[two heads]{l} \dots
\end{tikzcd}$$
where the maps are given by
\begin{eqnarray*}
\o_{2n_{i+1}}&\twoheadrightarrow&\o_{2n_i}\\
\begin{pmatrix}P_{11}&\dots&P_{1l_i}&\bullet&Q_{11}&\dots&Q_{1l_i}&\bullet\\\vdots&&\vdots&\vdots&\vdots&&\vdots&\vdots\\P_{l_i1}&\dots&P_{l_il_i}&\vdots&Q_{l_i1}&\dots&Q_{l_il_i}&\vdots\\\bullet&\dots&\dots&\bullet&\bullet&\dots&\dots&\bullet\\R_{11}&\dots&R_{1l_i}&\bullet&S_{11}&\dots&S_{1l_i}&\bullet\\\vdots&&\vdots&\vdots&\vdots&&\vdots&\vdots\\R_{l_i1}&\dots&R_{l_il_i}&\vdots&S_{l_i1}&\dots&S_{l_il_i}&\vdots\\\bullet&\dots&\dots&\bullet&\bullet&\dots&\dots&\bullet\end{pmatrix}&\mapsto&\begin{pmatrix}\sum_{k=1}^{l_i}P_{kk}&\sum_{k=1}^{l_i}Q_{kk}\\\sum_{k=1}^{l_i}R_{kk}&\sum_{k=1}^{l_i}S_{kk}\end{pmatrix}
\end{eqnarray*}
with $P_{k\ell}=-S_{\ell k}^T,Q_{k\ell},R_{k\ell}\in\gl_{n_i}$ such that $Q_{k\ell}+Q_{\ell k}^T=R_{k\ell}+R_{\ell k}^T=0$.

\begin{thm}\label{thm_mainod}
The space $V$ is $G$-Noetherian.
\end{thm}

This proof of this theorem will have the same structure as the proof of Theorem \ref{thm_mainsp}. Let $X\subsetneq V$ be a $G$-stable closed subset. Let $X_i$ be the closure of the projection of $X$ to $\o_{2n_i}$ and let $I(X_i)\subseteq K[\o_{2n_i}]$ be the ideal of $X_i$. If $\#\{i\mid l_i>1\}<\infty$, then Theorem \ref{thm_mainod} follows from \cite[Theorem 1.2]{eggermont-snowden}. So we assume that $\#\{i\mid l_i>1\}=\infty$. By restricting to an infinite subsequence, we may assume that $l_i\geq 3$ for all $i\in\NN$.

\begin{lm}\label{lm_e545ey3g}
Let $n\in\NN$, let $Y\subsetneq\o_{2n}$ be an $\O_{2n}$-stable closed subset and let $Z$ be the closed subset
$$
\left\{\begin{pmatrix}P&Q\\R&-P^T\end{pmatrix}\in\sp_{2n}~\middle|~P=P^T\right\}
$$
of $\o_{2n}$. Then there is a non-zero polynomial $f\in I(Y)$ whose top-graded part is not contained in the ideal of $Z$.
\end{lm}
\begin{proof}
The proof is analogous to the proof of Lemma \ref{lm_toppartnotinsp}.
\end{proof}

\begin{lm}\label{lm_d89fg7w98}
Let $i\in\NN$ and let $f=f(P,Q,R)\in I(X_i)$ be a non-zero polynomial whose top-graded part $g$ is not contained in the ideal of
$$
\left\{\begin{pmatrix}P&Q\\R&-P^T\end{pmatrix}\in\o_{2n_i}~\middle|~P=P^T\right\}.
$$
Then $I(X_{i+1})\cap K[r_{k\ell}|1\leq k,\ell\leq n_{i+1}]/(r_{k\ell}+r_{\ell k})$ contains a non-zero polynomial with degree at most $\deg(f)$.
\end{lm}
\begin{proof}
The proof is analogous to the proof of Lemma \ref{lm_movingeqsp2}, replacing $A(\lambda)$ by the matrix
$$
\begin{pmatrix}I_m&&&&\lambda I_m\\&I_m&&-\lambda I_m\\&&I_{n-2m}\\&&&I_m\\&&&&I_m\\&&&&&I_{n-2m}\end{pmatrix}\in\O_{2n}.
$$
\end{proof}

Since $X\subsetneq V$, we know that $X_j\subsetneq\o_{2n_j}$ for some $j\in\NN$. Using the previous lemma, we see that there is a $d\in\ZZ_{\geq0}$ such that $I(X_i)\cap K[r_{k\ell}\mid 1\leq k,\ell\leq n_i]/(r_{k\ell}+r_{\ell k})$ contains a non-zero polynomial of degree at most $d$ for all $i>j$. 

\begin{lm}\label{lm_boundedrankod}
Let $n\in\NN$, let $Y\subsetneq\o_{2n}$ be an $\O_{2n}$-stable closed subset and suppose that 
$$
I(Y)\cap K[r_{k\ell}|1\leq k,\ell\leq n]/(r_{k\ell}+r_{\ell k})
$$
contains a non-zero polynomial of degree $m+1$. Then
$$
Y\subseteq \left\{\begin{pmatrix}P&Q\\R&S\end{pmatrix}\in\o_{2n}~\middle|~\rk(Q),\rk(R)\leq 2m\right\}.
$$
Furthermore, if $n\geq 20m+2$, then $\rk(M)\leq 10m$ for all $M\in Y$.
\end{lm}
\begin{proof}
Let $Z$ be the closure of the subset 
$$
\left\{R~\middle|~\begin{pmatrix}P&Q\\R&S\end{pmatrix}\in Y\right\}
$$
of $\{R\in\gl_n\mid R+R^T=0\}$. Let $\GL_n$ act on $\o_{2n}$ via the diagonal embedding
\begin{eqnarray*}
\GL_n&\hookrightarrow&\O_{2n}\\
g&\mapsto&\Diag(g,g^{-T})
\end{eqnarray*}
and on $\{R\in\gl_n\mid R+R^T=0\}$ by $g\cdot R=gRg^T$. Then we see that $Y$ is $\GL_n$-stable and therefore $Z$ is also $\GL_n$-stable. So $Z$ must consist of all skew-symmetric matrices of rank at most $h$ for some even $h\leq n$. Since $I(Z)$ contains a non-zero polynomial of degree $m+1$, we see that $h\leq 2m$. See \cite[\S3]{abeasis-del-fra}. So 
$$
Y\subseteq\left\{\begin{pmatrix}P&Q\\R&S\end{pmatrix}\in\o_{2n}~\middle|~\rk(R)\leq 2m\right\}.
$$
Let $A\in\gl_n$ be a skew-symmetric matrix and let
$$
\begin{pmatrix}P&Q\\R&S\end{pmatrix}
$$
be an element of $Y$. Then we have
$$
\begin{pmatrix}0&I_n\\I_n&A\end{pmatrix}\in\O_{2n}
$$
and hence
$$
\begin{pmatrix}0&I_n\\I_n&A\end{pmatrix}\begin{pmatrix}P&Q\\R&S\end{pmatrix} \begin{pmatrix}0&I_n\\I_n&A\end{pmatrix}^{-1}=\begin{pmatrix}\bullet&\bullet\\Q+AS-PA-ARA&\bullet\end{pmatrix}\in Y.
$$
So we get $\rk(Q+AS-PA-ARA)\leq 2m$. Choosing $A=0$, we see that 
$$
Y\subseteq\left\{\begin{pmatrix}P&Q\\R&S\end{pmatrix}\in\o_{2n}~\middle|~\rk(Q)\leq 2m\right\}.
$$
Assume that $n\geq 2(3m+1)$. Since $S=-P^T$ and $A=-A^T$, we get
$$
\rk(PA-(PA)^T)\leq \rk(Q+AS-PA-ARA)+\rk(ARA)+\rk(Q)\leq 6m.
$$
Since $Y$ is $\GL_n$-stable, we have $\rk(P'A-(P'A)^T)\leq 6m$ for all $P'\sim P$. Choose 
$$
A=\begin{pmatrix}&&I_{3m+1}\\&0\\-I_{3m+1}\end{pmatrix}
$$ 
and write
$$
P'=\begin{pmatrix}P_{11}'&P_{12}'&P_{13}'\\P_{21}'&P_{22}'&P_{23}'\\P_{31}'&P_{32}'&P_{33}'\end{pmatrix}
$$
with $P_{11}',P_{13}',P_{31}',P_{33}'\in\gl_{3m+1}$. Then 
$$
P'A-(P'A)^T=\begin{pmatrix}\bullet&P_{23}'^T&\bullet\\-P_{23}'&0&P_{21}'\\\bullet&-P_{21}'^T&\bullet\end{pmatrix}
$$
has rank at most $6m$. Therefore the submatrix
$$
\begin{pmatrix}0&P_{21}'\\-P_{21}'^T&\bullet\end{pmatrix}
$$
also has rank at most $6m$ and hence and hence $\rk(P_{21}')\leq 3m$. By Proposition \ref{prop_tuplerank}, we see that $\rk(P,I_n)\leq 3m$. Hence
$$
Y\subseteq\{M\in\o_{2n}\mid \rk(M,\Diag(I_n,-I_n))\leq 2\cdot 2m+2\cdot 3m=10m\}.
$$
Assume that $n\geq 20m+2$, let $M+\lambda\Diag(I_n,-I_n)$ be an element of $Y$ with $\rk(M)\leq 10m$ and $\lambda\in K$ and let $B\in\gl_n$ be a skew-symmetric matrix of rank at least $n-1$. Then
$$
\begin{pmatrix}I_n&B\\&I_n\end{pmatrix}\in\O_{2n}
$$
and therefore
$$
\begin{pmatrix}I_n&B\\&I_n\end{pmatrix}\left(M+\lambda\Diag(I_n,-I_n)\right)\begin{pmatrix}I_n&B\\&I_n\end{pmatrix}^{-1}\in Y.
$$
So this element must be of the form $M'-\mu\Diag(I_n,-I_n)$ with $\rk(M)\leq 10m$ and $\mu\in K$. Now note that
$$
\rk\left(\lambda\begin{pmatrix}I_n&B\\&I_n\end{pmatrix}\Diag(I_n,-I_n)\begin{pmatrix}I_n&B\\&I_n\end{pmatrix}^{-1}+\mu\Diag(I_n,-I_n)\right)\leq \rk(M)+\rk(M')\leq 20m.
$$
So since
$$
\lambda\begin{pmatrix}I_n&B\\&I_n\end{pmatrix}\Diag(I_n,-I_n)\begin{pmatrix}I_n&B\\&I_n\end{pmatrix}^{-1}+\mu\Diag(I_n,-I_n)=\begin{pmatrix}\bullet&-2\lambda B\\\bullet&\bullet\end{pmatrix}
$$
and $\rk(2B)\geq n-1>20m$, we see that $\lambda=0$. Hence $Y$ consists of matrices of rank at most $10m$.
\end{proof}

Using Lemma \ref{lm_boundedrankod}, we see that there is an $m\in\ZZ_{\geq0}$ such that 
$$
X_i\subseteq\left\{\begin{pmatrix}P&Q\\R&S\end{pmatrix}\in\o_{2n}~\middle|~\rk(P)\leq m\right\}
$$
for all $i\gg0$. As in the proof of Lemma \ref{lm_boundedtuplerankimpliesrk=0}, we see using Lemma \ref{lm_higherrank} that this in fact holds for $m=0$.

\begin{lm}
Let $n\in\NN$ and let $Y\subsetneq\o_{2n}$ be an $\O_{2n}$-stable closed subset of 
$$
\left\{\begin{pmatrix}0&Q\\R&0\end{pmatrix}~\middle|~\begin{array}{c}Q\in\gl_n,Q+Q^T=0\\R\in\gl_n,R+R^T=0\end{array}\right\}.
$$
Then $Y\subseteq\{0\}$.
\end{lm}
\begin{proof}
Let 
$$
\begin{pmatrix}0&Q\\R&0\end{pmatrix}
$$
be an element of $Y$. Then 
$$
\begin{pmatrix}I_n&A\\&I_n\end{pmatrix}\begin{pmatrix}0&Q\\R&0\end{pmatrix}\begin{pmatrix}I_n&A\\&I_n\end{pmatrix}^{-1}=\begin{pmatrix}AR&\bullet\\\bullet&\bullet\end{pmatrix}\in Y
$$
for all $A\in\gl_n$ with $A+A^T=0$ since $Y$ is $\O_{2n}$-stable and therefore $R=0$. By Lemma~\ref{lm_boundedrankod}, we see that $Q=0$. 
\end{proof}

As in the previous section, the lemma shows that $X\subseteq\{0\}$. So again, when $\#\{i\mid l_i>1\}=\infty$, the only $G$-stable closed subsets of $V$ are $V$, $\{0\}$ and $\emptyset$ and the space $V$ is $G$-Noetherian.

\section{Limits of classical groups of type B}

In this last section of the proof of the Main Theorem, we still assume that $\cha(K)\neq2$. Now, we let $G$ be the direct limit of a sequence 
$$\begin{tikzcd}
\O_{2n_1+1} \arrow[hook,"\iota_1"]{r} & \O_{2n_2+1} \arrow[hook,"\iota_2"]{r} & \O_{2n_3+1} \arrow[hook,"\iota_3"]{r} & \dots
\end{tikzcd}$$
of diagonal embeddings. To prove that the corresponding inverse limit $V$ is $G$-Noetherian, it suffices to consider the case where $K$ is algebraically closed. The following proposition shows that, if $K=\overline{K}$ and $\iota_i$ has signature $(l_i,z_i)$ with $l_i$ even, then we can insert a group of type $D$ into the sequence defining $G$.

\begin{prop}
Suppose that $K$ is algebraically closed. Let $m,n\in\ZZ_{\geq0}$ be integers and let $\iota\colon\O_{2m+1}\hookrightarrow\O_{2n+1}$ be a diagonal embedding with signature $(l,z)$. If $l$ is even, then $\iota$ is the composition of diagonal embeddings $\O_{2m+1}\hookrightarrow\O_{l(2m+1)}$ and $\O_{l(2m+1)}\hookrightarrow\O_{2n+1}$.
\end{prop}
\begin{proof}
By Lemma \ref{lm_equiv_diag}, it suffices to find one diagonal embedding $\iota\colon\O_{2m+1}\hookrightarrow\O_{2n+1}$ with signature $(l,z)$ for which the proposition holds. For $k\in\NN$, note that the group
$$
H_k=\left\{A\in\GL_k~\middle|~A\begin{pmatrix}&&1\\&\iddots\\1\end{pmatrix}A^T=\begin{pmatrix}&&1\\&\iddots\\1\end{pmatrix}\right\}
$$
is conjugate to $\O_k$ in $\GL_k$. The map
\begin{eqnarray*}
H_{2m+1}&\hookrightarrow&H_{l(2m+1)}\\
A&\mapsto&\Diag(A,\dots,A)
\end{eqnarray*}
induces a diagonal embedding $\O_{2m+1}\hookrightarrow\O_{l(2m+1)}$ with signature $(l,0)$. Note that $2n+1=l(2m+1)+z$ and so $z$ is odd. Write $z=2k+1$. Then the map
\begin{eqnarray*}
\O_{l(2m+1)}&\hookrightarrow&\O_{2n+1}\\
\begin{pmatrix}A&B\\C&D\end{pmatrix}&\mapsto&\begin{pmatrix}A&&&B\\&I_k\\&&1\\C&&&D\\&&&&I_k\end{pmatrix}
\end{eqnarray*}
is a diagonal embedding with signature $(1,z)$. Now, let $\iota$ be the composition of these two diagonal embeddings. Then $\iota$ is itself a diagonal embedding and has signature $(l,z)$.
\end{proof}

Suppose that $K$ is algebraically closed and that the diagonal embeddings $\iota_i$ have signatures $(l_i,z_i)$ with $l_i$ even for infinitely many $i\in\NN$. Then the proposition shows that we can replace our sequence by a supersequence in which groups of type $D$ appear infinitely many times. In this case $V$ is $G$-Noetherian by the previous section. So, even if $K$ is not algebraically closed, we only have to consider the case where this does not happen. And, by replacing our sequence by an infinite subsequence, we may assume that $l_i\in\NN$ odd for every $i\in\NN$. As both $n_i$ and $n_{i+1}=l_in_i+z_i$ are odd, this forces $z_i\in\ZZ_{\geq0}$ to be even for all $i\in\NN$. Our next task is to find diagonal embeddings with such signatures.\bigskip

First, note that for $n\in\NN$ and $z\in\ZZ_{\geq0}$ the map
\begin{eqnarray*}
\iota_{1,2z}\colon \O_{2n+1}&\hookrightarrow&\O_{2(n+z)+1}\\
\begin{pmatrix}A&\alpha&B\\\beta&\mu&\gamma\\C&\delta&D\end{pmatrix}&\mapsto&\begin{pmatrix}A&&\alpha&B&\\&I_z&&&\\\beta&&\mu&\gamma&\\C&&\delta&D&\\&&&&I_z\end{pmatrix}
\end{eqnarray*}
is a diagonal embedding with signature $(1,2z)$. Here $A,B,C,D\in\gl_n$, $\alpha,\beta^T,\gamma^T,\delta\in K^n$ and $\mu\in K$. The associated map of Lie algebras is
\begin{eqnarray*}
\pr_{1,2z}\colon \o_{2(n+z)+1}&\twoheadrightarrow&\o_{2n+1}\\
\begin{pmatrix}P&\bullet&v&Q&\bullet\\\bullet&\bullet&\bullet&\bullet&\bullet\\\phi&\bullet&0&\psi&\bullet\\R&\bullet&w&S&\bullet\\\bullet&\bullet&\bullet&\bullet&\bullet\end{pmatrix}&\mapsto&\begin{pmatrix}P&v&Q\\\phi&0&\psi\\R&w&S\end{pmatrix}
\end{eqnarray*}
with $P=-S^T,Q,R\in\gl_n$ and $v=-\psi^T,w=-\phi^T\in K^n$ such that $Q+Q^T=R+R^T=0$.\bigskip

Next, we construct a diagonal embedding $\O_{2n+1}\hookrightarrow\O_{l(2n+1)}$ with signature $(l,0)$ for all $n\in\NN$ and $l\in\NN$ odd. Write
$$
J_k=\begin{pmatrix}&&1\\&\iddots\\1\end{pmatrix}\in\GL_k
$$
for $k\in\NN$ and take
$$
H_{2n+1,l}=\left\{A\in\GL_{l(2n+1)}~\middle|~A\begin{pmatrix}&&I_{ln}\\&J_l\\I_{ln}\end{pmatrix}A^T=\begin{pmatrix}&&I_{ln}\\&J_l\\I_{ln}\end{pmatrix}\right\}
$$
for all $n\in\NN$ and $l\in\NN$ odd. Then we have
$$
P\begin{pmatrix}&&I_{ln}\\&J_l\\I_{ln}\end{pmatrix}P^T=\begin{pmatrix}&&I_{ln+k}\\&1\\I_{ln+k}\end{pmatrix}
$$
where
$$
P=\begin{pmatrix}I_{ln}\\&I_k\\&&1\\&&&&I_{ln}\\&&&J_k\end{pmatrix}
$$
is a permutation matrix. So the map
\begin{eqnarray*}
H_{2n+1,l}&\rightarrow&\O_{l(2n+1)}\\
A&\mapsto&PAP^T
\end{eqnarray*}
is an isomorphism. Consider the map
\begin{eqnarray*}
\O_{2n+1}&\hookrightarrow&H_{2n+1,l}\\
\begin{pmatrix}A&\alpha&B\\\beta&\mu&\gamma\\C&\delta&D\end{pmatrix}&\mapsto&\begin{pmatrix}A&&&\alpha&&&&&B\\&\ddots&&&\ddots&&&\iddots\\&&A&&&\alpha&B\\\beta&&&\mu&&&&&\gamma\\&\ddots&&&\ddots&&&\iddots\\&&\beta&&&\mu&\gamma\\&&C&&&\delta&D\\&\iddots&&&\iddots&&&\ddots\\C&&&\delta&&&&&D\end{pmatrix}
\end{eqnarray*}
where $A,B,C,D\in\gl_n$, $\alpha,\beta^T,\gamma^T,\delta\in K^n$ and $\mu\in K$ all occur $l$ times on the right hand side. Write $l=2k+1$.  By taking the composition of these two maps, we get a diagonal embedding $\O_{2n+1}\hookrightarrow\O_{l(2n+1)}$ with signature $(l,0)$.\bigskip

Write $J=J_l$ and consider the Lie algebra
\begin{eqnarray*}
\h_{2n+1,l}&=&\left\{P\in\gl_{l(2n+1)}~\middle|~P\begin{pmatrix}&&I_{ln}\\&J_l\\I_{ln}\end{pmatrix}+\begin{pmatrix}&&I_{ln}\\&J_l\\I_{ln}\end{pmatrix}P^T=0\right\}\\
&=&\left\{\begin{pmatrix}P&V&Q\\\Phi&U&\Psi\\R&W&S\end{pmatrix}\in\gl_{l(2n+1)}~\middle|~\begin{array}{c}P+S^T=Q+Q^T=R+R^T=0\\VJ+\Psi^T=WJ+\Phi^T=0\\UJ+JU^T=0\end{array}\right\}
\end{eqnarray*}
of $H_{2n+1,l}$. The map $\O_{2n+1}\hookrightarrow H_{2n+1,l}$ corresponds to the map $\h_{2n+1,l}\twoheadrightarrow\o_{2n+1}$ sending
$$
\begin{pmatrix}P_{11}&\dots&P_{1l}&V_{11}&\dots&V_{1l}&Q_{11}&\dots&Q_{1l}\\\vdots&&\vdots&\vdots&&\vdots&\vdots&&\vdots\\P_{l1}&\dots&P_{ll}&V_{l1}&\dots&V_{ll}&Q_{l1}&\dots&Q_{ll}\\\Phi_{11}&\dots&\Phi_{1l}&U_{11}&\dots&U_{1l}&\Psi_{11}&\dots&\Psi_{1l}\\\vdots&&\vdots&\vdots&&\vdots&\vdots&&\vdots\\\Phi_{l1}&\dots&\Phi_{ll}&U_{l1}&\dots&U_{ll}&\Psi_{l1}&\dots&\Psi_{ll}\\R_{11}&\dots&R_{1l}&W_{11}&\dots&W_{1l}&S_{11}&\dots&S_{1l}\\\vdots&&\vdots&\vdots&&\vdots&\vdots&&\vdots\\R_{l1}&\dots&R_{ll}&W_{l1}&\dots&W_{ll}&S_{l1}&\dots&S_{ll}\end{pmatrix}
$$
to
$$\begin{pmatrix}P_{11}+\dots+P_{ll}&V_{11}+\dots+V_{ll}&Q_{1l}+\dots+Q_{l1}\\\Phi_{11}+\dots+\Phi_{ll}&U_{11}+\dots+U_{ll}&\Psi_{1l}+\dots+\Psi_{l1}\\R_{1l}+\dots+R_{l1}&W_{1l}+\dots+W_{l1}&S_{11}+\dots+S_{ll}\end{pmatrix}.
$$
Here, for each entry, we either sum along the diagonal or along the anti-diagonal in a manner consistent with the definition of the map $\O_{2n+1}\hookrightarrow H_{2n+1,l}$. The map $H_{2n+1,l}\rightarrow\O_{l(2n+1)}$ corresponds to the map $\o_{l(2n+1)}\rightarrow\h_{2n+1,l}$ sending $Q$ to $P^TQP^{-T}$. \bigskip

We let the diagonal embeddings in the sequence
$$\begin{tikzcd}
\O_{2n_1+1} \arrow[hook,"\iota_1"]{r} & \O_{2n_2+1} \arrow[hook,"\iota_2"]{r} & \O_{2n_3+1} \arrow[hook,"\iota_3"]{r} & \dots
\end{tikzcd}$$
be (compositions of) the forms above. As in the previous sections, if only finitely many embeddings have signature $(l_i,2z_i)$ with $l_i>1$, then Theorem \ref{thm_mainsp} follows from \cite[Theorem 1.2]{eggermont-snowden}. So we assume that $\#\{l_i\mid l_i>1\}=\infty$. Now, by replacing our sequence by an infinite subsequence, we may assume that $l_i\in\NN$ is odd and at least $3$ for every $i\in\NN$. 

\begin{lm}\label{lm_movingeqd2}
Let $Y\subsetneq\h_{2n+1,l}$ be an $H_{2n+1,l}$-stable closed subset and let $Z$ be the closed subset
$$
\left\{\begin{pmatrix}P&V&Q\\\Phi&U&\Psi\\R&W&S\end{pmatrix}\in\h_{2n+1,l}~\middle|~P=P^T\right\}
$$
of $\h_{2n+1,l}$. Then there is a non-zero polynomial $f\in I(Y)$ whose top-graded part is not contained in the ideal of $Z$.
\end{lm}
\begin{proof}
The proof is analogous to the proof of Lemma \ref{lm_toppartnotinsp}.
\end{proof}

\begin{lm}\label{lm_r489vysrs890}
Let $X$ be an $H_{2n+1,l}$-stable closed subset of $\h_{2n+1,l}$ and let $Y$ be the closure of its image in $\o_{2n+1}$. Let $f\in I(Y)\subseteq K[\o_{2n+1}]$ be a non-zero polynomial whose top-graded part $g$ is not contained in the ideal of
$$
\left\{\begin{pmatrix}P&V&Q\\\Phi&U&\Psi\\R&W&S\end{pmatrix}\in\h_{2n+1,l}~\middle|~P=P^T\right\}.
$$
Then $I(X)$ contains a non-zero polynomial with degree at most $\deg(f)$ that only depends on $R$ and two columns of $W$.
\end{lm}
\begin{proof}
Consider the matrix
$$
\begin{pmatrix}P_{11}&\dots&P_{1l}&V_{11}&\dots&V_{1l}&Q_{11}&\dots&Q_{1l}\\\vdots&&\vdots&\vdots&&\vdots&\vdots&&\vdots\\P_{l1}&\dots&P_{ll}&V_{l1}&\dots&V_{ll}&Q_{l1}&\dots&Q_{ll}\\\Phi_{11}&\dots&\Phi_{1l}&U_{11}&\dots&U_{1l}&\Psi_{11}&\dots&\Psi_{1l}\\\vdots&&\vdots&\vdots&&\vdots&\vdots&&\vdots\\\Phi_{l1}&\dots&\Phi_{ll}&U_{l1}&\dots&U_{ll}&\Psi_{l1}&\dots&\Psi_{ll}\\R_{11}&\dots&R_{1l}&W_{11}&\dots&W_{1l}&S_{11}&\dots&S_{1l}\\\vdots&&\vdots&\vdots&&\vdots&\vdots&&\vdots\\R_{l1}&\dots&R_{ll}&W_{l1}&\dots&W_{ll}&S_{l1}&\dots&S_{ll}\end{pmatrix}\in\h_{2n+1,l}
$$
and note that the polynomial $f=f(P,Q,R,v,w)\in I(Y)$ induces the element 
$$
f(P_{11}+\dots+P_{ll},Q_{1l}+\dots+Q_{l1},R_{1l}+\dots+R_{l1},V_{11}+\dots+V_{ll},W_{1l}+\dots+W_{l1})
$$
of $I(X)$. Consider the matrix
$$
A(\lambda)=\begin{pmatrix}I_n&&&&&&-\lambda I_n\\&\ddots\\&&I_n&&\lambda I_n\\&&&I_l\\&&&&I_n\\&&&&&\ddots\\&&&&&&I_n\end{pmatrix}\in H_{2n+1,l}
$$
for $\lambda\in K$. One can check that
$$
g(R_{1l}-R_{l1},-(R_{1l}+R_{l1}),R_{1l}+\dots+R_{l1},W_{1l}-W_{l1},W_{1l}+\dots+W_{l1})
$$
is contained in the span of
$$
A(\lambda)\cdot f(P_{11}+\dots+P_{ll},Q_{1l}+\dots+Q_{l1},R_{1l}+\dots+R_{l1},V_{11}+\dots+V_{ll},W_{1l}+\dots+W_{l1})
$$
over all $\lambda\in K$. So it is an element of $I(X)$ and its degree is at most $\deg(f)$.\bigskip

Next, consider the matrix
$$
B(\mu)=\begin{pmatrix}I_{ln}\\&1\\&\mu&\ddots\\&&&\ddots\\&&&-\mu&1\\&&&&&I_{ln}\end{pmatrix}\in H_{2n+1,l}
$$
for $\mu\in K$. Let $h(P,Q,R,v,w)$ be the top-graded part of $g$ with respect to the grading where $P,Q,R$ get grading $0$ and $v,w$ get grading $1$. Then one can check that
$$
h(R_{1l}-R_{l1},-(R_{1l}+R_{l1}),R_{1l}+\dots+R_{l1},-W_{l-1,2},W_{1l}+W_{l-1,2})
$$
is contained in the span of
$$
B(\mu)\cdot g(R_{1l}-R_{l1},-(R_{1l}+R_{l1}),R_{1l}+\dots+R_{l1},W_{1l}-W_{l1},W_{1l}+\dots+W_{l1})
$$
over all $\mu\in K$. This polynomial is contained in $I(X)$ and has degree at most $\deg(f)$. 
\end{proof}

The following proposition tells us how to use the equation we gain from Lemma~\ref{lm_movingeqd2}. Let $\GL_n$ act on $\{Q\in\gl_n\mid Q=-Q^T\}$ by $g\cdot Q=gQg^T$. Let $k\leq n$ be an integer and let $\GL_n$ act on $K^{n\times k}$ by left-multiplication.

\begin{prop}\label{prop_boundedrankd}
Let $R\in\gl_n$ be a skew-symmetric matrix and let $W\in K^{n\times k}$ be a matrix of rank $k$. Then the closure of the $\GL_n$-orbit of $(R,W)$ inside $\{Q\in\gl_n\mid Q=-Q^T\}\oplus K^{n\times k}$ contains all tuples $(Q,V)$ with $\rk(Q)\leq\rk(R)-2k$. 
\end{prop}
\begin{proof}
We will prove the proposition using induction on $k$. The case $k=0$ is well-known. So assume that $0<2k\leq\rk(R)$. Let $X$ be the closure of the $\GL_n$-orbit of $(R,W)$. Note that we may replace $(R,W)$ with any element in its $\GL_n$-orbit. Since $\rk(W)=k$, we may therefore assume that the last column of $W$ equals $e_n$. Now, if we act with a matrix of the form 
$$
\begin{pmatrix}1\\&\ddots\\&&\ddots\\a_1&\dots&a_{n-1}&1\end{pmatrix},
$$
then the last column of $W$ stays equal to $e_n$. And, the last column of $R$ becomes 
$$
\binom{a_1r_1+\dots+a_{n-1}r_{n-1}+r_n}{0}
$$
if we write
$$
R=\begin{pmatrix}
r_1&\dots&r_{n-1}&r_n\\
\bullet&\dots&\bullet&0
\end{pmatrix}
$$
with $r_1,\dots,r_n\in K^{n-1}$. As $\rk(R)>k=\rk(W)$ and $e_n$ is contained in the image of $W$, we see that 
$$
\binom{a_1r_1+\dots+a_{n-1}r_{n-1}+r_n}{0}
$$ 
is not contained in the image of $W$ for some $a_1,\dots,a_{n-1}$. So we may also assume that the last column of $R$ is not contained in the image of $W$. Next, note that the last column of $W$ stays $e_n$ and the last column of $R$ stays outside the image of $W$ if we act with a matrix of the form $\Diag(g,1)$ with $g\in\GL_{n-1}$. Since the last column of $R$ is non-zero, we may therefore assume in addition that 
$$
R=\begin{pmatrix}R'&w&0\\-w^T&0&1\\0&-1&0\end{pmatrix}
$$
for some $R'\in\gl_{n-2}$ and $w\in K^{n-2}$. So the vector $e_{n-1}$ is not contained in the image of $W$. Note that $\rk(R')\geq\rk(R)-2$. Write
$$
W=\begin{pmatrix}W'&0\\v^T&0\\u^T&1\end{pmatrix}
$$
with $W'\in K^{(n-2)\times(k-1)}$ and $u,v\in K^{k-1}$. Since $e_{n-1}$ is not contained in the image of $W$, the matrix $(W~e_{n-1})$ has rank $k+1$ and hence $\rk(W')=k-1$. The limit
$$
\lim_{\lambda\rightarrow0}\Diag(I_{n-2},\lambda,1)\cdot(R,W)=\left(\begin{pmatrix}R'&0&0\\0&0&0\\0&0&0\end{pmatrix},\begin{pmatrix}W'&0\\0&0\\u^T&1\end{pmatrix}\right)
$$
is an element of $X$. Using the induction hypothesis, we see that $X$ contains
$$
\left(\begin{pmatrix}Q&0&0\\0&0&0\\0&0&0\end{pmatrix},\begin{pmatrix}V&0\\0&0\\u^T&1\end{pmatrix}\right)
$$
for all skew-symmetric matrices $Q\in\gl_{n-2}$ of rank at most $\rk(R)-2k$ and all $V\in K^{(n-2)\times(k-1)}$. By acting with a permutation matrix, we see in particular that
$$
\left(\begin{pmatrix}Q&0&0\\0&0&0\\0&0&0\end{pmatrix},\begin{pmatrix}0&0\\I_{k-1}&0\\u^T&1\end{pmatrix} \right)\in X
$$
for all skew-symmetrix matrices $Q\in\gl_{n-k}$ of rank at most $\rk(R)-2k$. Therefore
$$
\left(\Diag(Q,0),V\right)=\lim_{\lambda\rightarrow0}\left(\Diag(I_{n-k},\lambda I_k)+\left(0~V\begin{pmatrix}I_{k-1}&0\\-u^T&1\end{pmatrix} \right)\right)\cdot \left(\begin{pmatrix}Q&0&0\\0&0&0\\0&0&0\end{pmatrix},\begin{pmatrix}0&0\\I_{k-1}&0\\u^T&1\end{pmatrix} \right)\in X
$$
for all skew-symmetrix matrices $Q\in\gl_{n-k}$ of rank at most $\rk(R)-2k$ and all matrices $V\in K^{n\times k}$. So since $X$ is $\GL_n$-stable, we see that $(Q,V)\in X$ for all skew-symmetric matrices $Q\in\gl_n$ of rank at most $\rk(R)-2k$ and all matrices $V\in K^{n\times k}$.
\end{proof}

\begin{lm}\label{lm_4t5789gyr89}
There are integers $c_0,c_1,c_2\in\NN$ such that the following holds: let $m\in\ZZ_{\geq0}$ be an integer with $c_2m\leq n$ and let $M\in\h_{2n+1,l}$ be an element such that for all matrices
$$
\begin{pmatrix}P&V&Q\\\Phi&U&\Psi\\R&W&S\end{pmatrix}\in H_{2n+1,l}\cdot M
$$
it holds that $\rk(R)\leq m$ or the first and last column of $W$ are linearly dependent. Then we have $\rk(M)\leq c_1m+c_0$.
\end{lm}
\begin{proof}
Let
$$
\begin{pmatrix}P&V&Q\\\Phi&U&\Psi\\R&W&S\end{pmatrix}
$$
be an element of the orbit of $M$. We assume that $c_2m\leq n$ with $c_2$ high enough and we will prove a series of claims, which together imply that $\rk(M)\leq c_1m+c_0$ for suitable $c_0,c_1\in\NN$.

\begin{itemize}
\item[(x)] We have $\rk(R)\leq m+4$.
\end{itemize}
Suppose that $\rk(R)>m$. Note that $\Diag(I_{ln},g,I_{ln})\in H_{2n+1,l}$ for all $g\in\GL_l$ with $gJg^T=J$. We have
$$
\begin{pmatrix}I_{ln}\\&g\\&&I_{ln}\end{pmatrix}\begin{pmatrix}P&V&Q\\\Phi&U&\Psi\\R&W&S\end{pmatrix}\begin{pmatrix}I_{ln}\\&g\\&&I_{ln}\end{pmatrix}^{-1}=\begin{pmatrix}P&Vg^{-1}&Q\\g\Phi&gUg^{-1}&g\Psi\\R&Wg^{-1}&S\end{pmatrix}
$$
for all $g\in\GL_l$. So we see that the first and last column of $Wg^{-1}$ are linearly dependent for all $g\in\GL_l$ with $gJg^T=J$. Using the fact that
$$
g=\begin{pmatrix}1\\&\ddots\\\lambda&&\ddots\\&&&\ddots\\&&-\lambda&&1\end{pmatrix}
$$ 
satisfies $gJg^T=J$ as long as $\lambda$ is not in the middle row together with $JJJ^T=J$, it is now easy to check that $\rk(W)\leq 2$. Next, note that
$$
\begin{pmatrix}I_{ln}&A&-\frac{1}{2}AJA^T\\&I_l&-JA^T\\&&I_{ln}\end{pmatrix}\in H_{2n+1,l}
$$
for all $A\in K^{ln\times l}$. For all $A\in K^{ln\times l}$, we have
$$
\begin{pmatrix}I_{ln}&A&-\frac{1}{2}AJA^T\\&I_l&-JA^T\\&&I_{ln}\end{pmatrix}^{-1}\begin{pmatrix}P&V&Q\\\Phi&U&\Psi\\R&W&S\end{pmatrix}\begin{pmatrix}I_{ln}&A&-\frac{1}{2}AJA^T\\&I_l&-JA^T\\&&I_{ln}\end{pmatrix}=\begin{pmatrix}\bullet&\bullet&\bullet\\\bullet&\bullet&\bullet\\R&W+RA&\bullet\end{pmatrix}
$$
and hence $\rk(W+RA)\leq 2$. So $\rk(RA)\leq 4$ and hence $\rk(R)\leq 4$.

\begin{itemize}
\item[(y)] We have $\rk(Q)\leq m+4$ and $\rk(P)=\rk(S)\leq 3(m+4)/2$.
\end{itemize}
Repeat the proof of Lemma \ref{lm_boundedrankod} and act with matrices
$$
\begin{pmatrix}&&I_{ln}\\&I_l\\I_{ln}&&A\end{pmatrix},\begin{pmatrix}I_{ln}&&B\\&I_l\\&&I_{ln}\end{pmatrix}
$$
with $A=-A^T$ and $B=-B^T$.

\begin{itemize}
\item[(z)] We have $\rk(W)=\rk(\Phi),\rk(V)=\rk(\Psi)\leq 4(m+4)$ and $\rk(U)\leq22(m+4)$.
\end{itemize}
We have 
$$
\begin{pmatrix}I_{ln}&A&-\frac{1}{2}AJA^T\\&I_l&-JA^T\\&&I_{ln}\end{pmatrix}^{-1}\begin{pmatrix}P&V&Q\\\Phi&U&\Psi\\R&W&S\end{pmatrix}\begin{pmatrix}I_{ln}&A&-\frac{1}{2}AJA^T\\&I_l&-JA^T\\&&I_{ln}\end{pmatrix}=\begin{pmatrix}\bullet&\bullet&\bullet\\\bullet&\bullet&\bullet\\\bullet&\bullet&T\end{pmatrix}
$$
with $T=-\frac{1}{2}RAJA^T-WJA^T+S$ for all $A\in K^{ln\times l}$. So $\rk(WJA^T)\leq 4(m+4)$ for all $A\in K^{ln\times l}$. So $\rk(W)=\rk(\Phi)\leq 4(m+4)$. By conjugating with
$$
\begin{pmatrix}&&I_{ln}\\&I_l\\I_{ln}\end{pmatrix}
$$  
we also see that $\rk(V)=\rk(\Psi)\leq 4(m+4)$. We have
$$
\begin{pmatrix}I_{ln}&A&-\frac{1}{2}AJA^T\\&I_l&-JA^T\\&&I_{ln}\end{pmatrix}^{-1}\begin{pmatrix}P&V&Q\\\Phi&U&\Psi\\R&W&S\end{pmatrix}\begin{pmatrix}I_{ln}&A&-\frac{1}{2}AJA^T\\&I_l&-JA^T\\&&I_{ln}\end{pmatrix}=\begin{pmatrix}\bullet&\bullet&T\\\bullet&\bullet&\bullet\\\bullet&\bullet&\bullet\end{pmatrix}
$$
with
$$
T=\begin{pmatrix}I_{ln}&-A&-\frac{1}{2}AJA^T\end{pmatrix}\begin{pmatrix}P&V&Q\\\Phi&U&\Psi\\R&W&S\end{pmatrix}\begin{pmatrix}-\frac{1}{2}AJA^T\\-JA^T\\I_{ln}\end{pmatrix}.
$$
Now, we know that $\rk(T)\leq m+4$. Also, the matrix $T$ is a sum of nine matrices: the matrix $AUJA^T$ and eight other matrices for which we have found bounds on the rank. Adding all these bounds together, we find that 
$$
\rk(AUJA^T)\leq(1+1+1+3/2+3/2+4+4+4+4)(m+4)=22(m+4)
$$ for all $A\in K^{ln\times l}$. Hence $\rk(U)\leq 22(m+4)$. \bigskip

Together (x), (y) and (z) show that
$$
\rk\begin{pmatrix}P&V&Q\\\Phi&U&\Psi\\R&W&S\end{pmatrix}\leq c_1m+c_0
$$
for some $c_0,c_1\in\NN$. So this holds in particular if we let this matrix be $M$ itself.
\end{proof}

We combine these results as in the previous section. Lemmas \ref{lm_movingeqd2} and \ref{lm_r489vysrs890} play the roles of Lemmas \ref{lm_e545ey3g} and \ref{lm_d89fg7w98} and give us off-diagonal polynomials. Then, Proposition \ref{prop_boundedrankd} with $k=2$ shows us the structure of the off-diagonal part of the matrix as a $\GL_n$-representation with the Zariski topology. From this and the degree of the off-diagonal polynomial, we get bounds on ranks of some submatrices. Lemma \ref{lm_4t5789gyr89} turns these bounds into a rank bound on the matrix itself. Finally, we find similarly to Lemma \ref{lm_boundedtuplerankimpliesrk=0} that $X\subseteq\{0\}$ and this implies that $V$ is $G$-Noetherian.

\section{Further questions}

\subsection*{Representation-inducing functors}
As stated in the introduction, many examples of infinite-dimensional spaces that are Noetherian up to the action of some group arise from taking limits of sequences after applying certain functors. So one could hope that our spaces $V$ and groups $G$ can be contructed from functors in such a way that these functors are suitably Noetherian and that this Noetherianity implies the results of this paper. Concretely, is there a class of topologically Noetherian functors from which the representations in this paper arise and do any new representations arise from such functors? 

\subsection*{Classifications for types \texorpdfstring{$\B$, $\C$ and $\D$}{B, C and D}}
Theorem \ref{thm_mainsl} classifies all $G$-stable closed subsets of $V$ when $G$ is the direct limit of diagonal embeddings between classical groups of type $\A$. One wonders whether such a classification exists for the other types. The key part of the proof of Theorem \ref{thm_mainsl} seems to be Proposition \ref{prop_pointclosure}, which gives a complete descriptions of the closures of orbits. So it would be very interesting to see whether such descriptions can be found for the other types.

\end{document}